\documentclass[preprint,1p]{elsarticle}

\usepackage{amssymb,amsmath,amsthm}
\usepackage[numbers]{natbib}
\usepackage{geometry}
\usepackage{fleqn}
\usepackage{graphicx}
\usepackage{txfonts}
\usepackage{hyperref}
\usepackage{enumitem}
\usepackage{adjustbox}
\usepackage{changepage}
\usepackage[toc,page]{appendix}
\usepackage{subcaption}
\usepackage{multirow}
\usepackage{float}

\usepackage{xspace}
\usepackage{textcomp}
\usepackage{url}
\usepackage{mathbbol}
\usepackage{algorithm}
\usepackage{algpseudocode}
\usepackage{multirow}
\usepackage{bm}
\usepackage[Symbol]{upgreek}

\usepackage{MacrosAndBiblio/macros}

\usepackage[dvipsnames]{xcolor}

\usepackage[normalem]{ulem}

\newcommand{\eg}{e.g.,\@\xspace}
\newcommand{\ie}{i.e.,\@\xspace}

\begin{document}

\begin{frontmatter}
  \title{Structure preserving numerical methods for the ideal compressible MHD system}
  \tnotetext[t1]{This research is funded by Swedish Research Council (VR) under grant number 2021-04620.}
  \author[1]{Tuan Anh Dao}
  \ead{tuananh.dao@it.uu.se}
  \author[1]{Murtazo Nazarov}
  \ead{murtazo.nazarov@it.uu.se}
  \author[2]{Ignacio Tomas}
  \ead{igtomas@ttu.edu}
  \address[1]{Department of Information Technology, Uppsala University, Sweden}
  \address[2]{Department of Mathematics \& Statistics
    Texas Tech University}

\begin{abstract}
We introduce a novel structure-preserving method in order to approximate
the compressible ideal Magnetohydrodynamics (MHD) equations. This
technique addresses the MHD equations using a non-divergence formulation, where
the contributions of the magnetic field to the momentum and total mechanical
energy are treated as source terms. Our approach uses the Marchuk-Strang
splitting technique and involves three distinct components: a compressible
Euler solver, a source-system solver, and an update procedure for the total
mechanical energy. The scheme allows for significant freedom on the choice of
Euler's equation solver, while the magnetic field is discretized using a
curl-conforming finite element space, yielding exact preservation of the
involution constraints. We prove that the method preserves invariant domain
properties, including positivity of density, positivity of internal energy,
and the minimum principle of the specific entropy. If the scheme used to solve
Euler's equation conserves total energy, then the resulting MHD scheme can be
proven to preserve total energy. Similarly, if the scheme used to solve Euler's
equation is entropy-stable, then the resulting MHD scheme is entropy stable as
well. In our approach, the CFL condition does not depend on magnetosonic
wave-speeds, but only on the usual maximum wavespeed from Euler's system. To
validate the effectiveness of our method, we solve a variety of ideal MHD
problems, showing that the method is capable of delivering high-order accuracy
in space for smooth problems, while also offering unconditional robustness in
the shock hydrodynamics regime as well.
\end{abstract}

\begin{keyword}
  MHD \sep structure preserving \sep invariant domain  \sep involution
constraints \sep energy-stability
\end{keyword}

\end{frontmatter}

\section{Introduction}
Magnetohydrodynamic (MHD) equations model the dynamics of plasma, which is an
ionizied gas at high temperatures. The ideal MHD equations combines: the fluid
dynamics equations description of Euler equations, with zero-permittivity limit
of Maxwell's equations and Ohm's law closure. The model considers both the
movement of the conductive fluid and its interaction with magnetic fields. The
MHD equations are widely used in astrophysics applications as well as in
nuclear fusion research, where it is used to study and control instabilities in
the plasma confinement.

Solutions to the MHD system contain contact, shock, and rarefaction
waves. In addition, the interaction of fluid and magnetic fields at very high
temperatures pose additional challenges for MHD simulations. Despite these
difficulties, numerical solutions of the MHD system are vital to predict
phenomena in various scientific fields such as plasma physics and
astrophysics. Furthermore, when performing numerical simulations of the MHD
system, it is crucial to ensure the preservation of essential
structure of the solution, such as positivity properties, conservation of total
energy, and involution constraints.

Various schemes that retain several of these properties, in particular for the
compressible Euler equations, are published in existing literature. For
instance, the works of \cite{Perthame1996, Zhang2010, Guer2018}, along with
references provided therein, represent just a subset of the comprehensive
research dedicated to achieving positivity-preserving approximations for the
compressible Euler equations, using finite volume, discontinuous Galerkin, and
finite element methods. Unfortunately, direct extension of these methods for
the MHD system is not straighforward due to the additional induction equation
for the magnetic field and corresponding to the magnetic stress/force. In
particular, the standard MHD model in divergence form is only valid if
$\diver{} \Bfield \equiv 0$ at all times. A slight violation of the
divergence-free condition can lead to negative internal energy, which will
cause the numerical simulation to fail catastrophically, see \eg \cite{Wu2018,
Wu2019}. It should be emphasized that the divergence formulation of the MHD
system is valid only for sufficiently smooth solutions. However, in the
case of weakly differentiable and discontinuous solutions, $\diver{} \Bfield$
cannot be pointwise zero. To the best of our knowledge, none of the divergence
cleaning techniques, such as \cite{Dedner2002}, can completely eliminate the
discrepancy error of the divergence of $\Bfield$.

In this paper, instead of using the MHD equations in divergence form, see
\eqref{MHDsystemDiv}, which is widely used in scientific works, we proposed to
use the induction equation and preserve the magnetic forces acting on the
momentum and total mechanical energy as source terms. More precisely, we
propose solving
\begin{subequations}
  \label{MHDsystem}
  \begin{align}
    \label{consMass}
    \partial_t \dens + \diver{}\mom &= 0 \, , \\
    \partial_t \mom + \diver{}(\dens^{-1} \mom \mom^\transp
    + \mathbb{I} p) &=
   - \mu  \Hfield \times  \curl{}\Hfield  \, , \\
    \label{consMechE}
    \partial_t \totme + \diver{}\big(\tfrac{\mom}{\dens} (E + p) \big) &=
    - \mu (\Hfield \times  \curl{}\Hfield) \cdot \tfrac{\mom}{\dens} \, , \\
    \label{consBfield}
    \partial_t \Hfield - \curl{} (\tfrac{\mom}{\dens} \times \Hfield) &= 0 \, ,
  \end{align}
\end{subequations}
where $\rho$ is the density, $\mom$ is the momentum, $\totme =
\frac{1}{2 \rho}|\mom|^2 + \rho e$ is the total mechanical energy, $e$ is the
specific internal energy, $\Hfield$ is the magnetic field, $p = p(\rho, e)$ is
the pressure, $\mathbb{I} \in \mathbb{R}^{d\times d}$ denotes the identity
matrix with $d$ being the space dimension, and $\mu>0$ is the
magnetic permeability constant. Taking the divergence to both sides of
\eqref{consBfield} we obtain the condition
$
\partial_t \diver{}\Hfield = 0,
$
implying that the evolution of the magnetic field $\Hfield$ is such
that $\diver{}\Hfield(t) \equiv \diver{}\Hfield_0$ for all $t \geq 0$, where
$\Hfield_0$ is the initial data.

Note, that for the case of smooth (\eg $\mathcal{C}^1$-continuous or better)
divergence-free solutions the formulation \eqref{MHDsystem} is equivalent to
the MHD system in the divergence form \eqref{MHDsystemDiv}. However, for the
case of weakly-differentiable and discontinuous solutions, we should regard
\eqref{MHDsystem} and \eqref{MHDsystemDiv} as entirely different models. In
particular, there is no reason to believe that \eqref{MHDsystem} and
\eqref{MHDsystemDiv} should produce the same families of discontinuous
solutions, see for instance \cite[p. 253]{Smoller1994} on a related discussion.
We emphasize that formulation \eqref{MHDsystemDiv} is not valid without the
assumption $\diver{}\Bfield = \diver{}\Bfield_0 = 0$ since it is an intrinsic
part of its derivation. On other hand, formulation \eqref{MHDsystem} does not
need or use the condition $\diver{}\Hfield \equiv 0 $: there is no
mathematical reason to incorporate such assumption.

From a practical point of view, regardless of whether we prefer
source-formulation \eqref{MHDsystem} or divergence formulation
\eqref{MHDsystemDiv}, any numerical method satisfying the following:
\begin{itemize}
\item[\textrm{(i)}] Preservation of pointwise stability properties: such as
  pointwise positivity of the density and minimum principle of the specific
  entropy;
\item[\textrm{(ii)}] Preservation of involution constraints, in this case,
  preservation of the weak-divergence;
\item[\textrm{(iii)}] Preservation of total energy;
\item[\textrm{(iv)}] Preservation of second order accuracy (or higher) for
  smooth solutions;
\item[\textrm{(v)}] Preservation of discrete entropy-dissipation properties.
\end{itemize}
is a desirable method for engineering and scientific applications. List
\textrm{(i)}-\textrm{(v)} is quite ambitious and we are not aware of any
numerical scheme capable of preserving properties \textrm{(i)-(v)}
simultaneously. We highlight that designing a scheme that
preserves just one of these properties (\eg formal high-order accuracy,
see for instance \cite{Dao2022a, Dao2022b}) does not pose a major
challenge. The mathematical challenge of structure preservation
lies in the satisfaction of two or more of these properties simultaneously. In
this manuscript, we advocate for the use of formulation \eqref{MHDsystem},
instead of the usual divergence form \eqref{MHDsystemDiv}, as better fit in
order to preserve properties (i)-(v) outlined above.

In \cite{Dao2022a} it was proved that by adding a viscous term to each equation
of the ideal MHD system (\ie conservation of mass, conservation of momentum,
conservation of total energy, and the induction equation) one can achieve
positivity of density and internal energy, minimum principle of the specific
entropy, and satisfaction of all generalized entropies. In this article we
improve the result of \cite{Dao2022a}. We prove that the viscous regularization
of mass, momentum, and total mechanical energy is sufficient to achieve the
above mentioned properties (\ie positivity of density and internal
energy, minimum principle of the specific entropy, and compatibility with all
generalized entropies). This shows that there is no need to regularize the
induction equation. This is a rather puzzling result, hinting at the
idea that the inclusion of the $\Bfield$ field in the MHD Riemann problem is an
artificial construct. We propose to separate the evolution equation of $\Bfield$
from the other components of the system (density, momentum, and total mechanical
energy), as originally described by the non-divergence formulation
\eqref{MHDsystem}. This is by no means a new idea: treating $\Bfield$
independently using its own spatial discretization has been proposed, for
example, in \cite{Mishra2011, Pagliantini2018, Fuchs2009} and references
therein. However, our approach still represents a major departure from
previously existing ideas and methods for the MHD system:
\begin{itemize}
\item[\thickdash] The induction equation is not treated as an isolated object,
but rather as a constituent of a Hamiltonian system consisting in: the
  balance of momentum subject to the Lorentz force $\mu (\curl{}\Hfield
  \times\Hfield)$, coupled to the induction equation \eqref{consBfield}, see for
  instance expressions \eqref{Operator2} and \eqref{SourceSmoothTest}. By
  treating a Hamiltonian system as such, this lends itself to natural
  definitions of stability that we can preserve in the fully-discrete setting.
\item[\thickdash] We use no advective stabilization in any form or fashion for
  the induction equation. This is tacitly suggested by the viscous regularization
  argument indicating that no artificial viscosity is required for the magnetic
  field $\Hfield$. This is also consistent with Hamiltonian systems, such as
  \eqref{SourceSmoothTest}, where the natural notion of stability is preservation
  of quadratic invariants. We avoid construing the induction equation as an
  advective system \cite{Heumann2013, Heumann2016}, Friedrich's system
  \cite{Besse2005}, or vanishing-viscosity limit (\eg conservation law).
\item[\thickdash] We use a primal (no vector potential) curl-conforming
  framework in order to discretize the magnetic field $\Hfield$. This is
  consistent with the preservation of weak divergence. We do not pursue to
  preserve a zero strong-divergence, or use a div-conforming framework as
  suggested for instance in \cite{Basting2017, Pagliantini2018}. However, we
  show that the method can preserve zero weak-divergence to machine accuracy
  for smooth as well as non-smooth regimes. We use no divergence cleaning.
\item[\thickdash] Energy of the non-divergence system \eqref{MHDsystem} is
  defined by a functional that consists in the sum of a linear $+$
  quadratic functional, see Section \ref{sec:energy}. This is quite different
  from the case of the divergence system \eqref{MHDsystemDiv} where energy
  stability consists in preserving the property
  $\int_{\Omega} \tote(t) \dx = \int_{\Omega} \tote_0 \dx$, which is the
  preservation of a linear functional.

\item[\thickdash] The resulting scheme preserves properties (i)-(iv) outlined
  above. This scheme can be used for smooth as well as extreme
  shock-hydrodynamics regimes. Property (v), entropy stability, can be preserved
  as well, provided the hyperbolic solver used to discretize Euler's subsystem is
  entropy stable. We make no emphasis on property (v) since there the is a very
  large literature on the matter. The scheme runs at the CFL of Euler's
  system, with no time-step size restriction due to magnetosonic waves. There
  is, in principle, no limit on the formal spatial accuracy of the scheme.
\end{itemize}

The outline of the paper is as follows: in Section \ref{sec:MHDprop} we
provide all the necessary background in relationship to the mathematical
properties of the MHD and Euler's system. In
Sections \ref{sec:SpaceDisc}-\ref{sec:SourceUpdate} we summarize the main
properties of the spatial and temporal discretizations that will be used. In
Section \ref{sec:FirstScheme} we present the scheme and make precise its
mathematical properties. Finally, in Section \ref{sec:numexp} we present
numerical results illustrating the efficiency of the solver is the context of
smooth as well as non-smooth test problems. We highlight that the main ideas
advanced in this paper can be implemented using quite general hyperbolic solvers
for Euler's equation. In Section \ref{sec:FirstScheme} we outline the structure
and mathematical properties expected from such hyperbolic solver. For the sake
of completeness we also describe the hyperbolic solvers used for all
computations in \ref{sec:HypSolver}.

%%%%%%%%%%%%%%%%%%%%%%%%%%%%%%%%%%%%%%%%%%%%%%%%%%%%%%%%%%%%%%%%%%%%%%%%%%%%%%%% 
%%%%%%%%%%%%%%%%%%%%%%%%%%%%%%%%%%%%%%%%%%%%%%%%%%%%%%%%%%%%%%%%%%%%%%%%%%%%%%%% 
%%%%%%%%%%%%%%%%%%%%%%%%%%%%%%%%%%%%%%%%%%%%%%%%%%%%%%%%%%%%%%%%%%%%%%%%%%%%%%%% 
%%%%%%%%%%%%%%%%%%%%%%%%%%%%%%%%%%%%%%%%%%%%%%%%%%%%%%%%%%%%%%%%%%%%%%%%%%%%%%%% 
\section{Main properties of the MHD system}\label{sec:MHDprop}

\subsection{Vanishing-viscosity limits and invariant sets}\label{sec:viscous}
In this section we improve the result of \cite{Dao2022a}. Let us consider the
case where initial magnetic field is divergence-free, \ie $\diver{}
\Bfield_0 = 0$. As already mentioned in the introduction, this implies that
$\diver{} \Bfield = 0$ also for all time $t$. Therefore, the system
\eqref{MHDsystem} can be re-written in the following divergence form:
\begin{subequations}\label{MHDsystemDiv}
  \begin{align}
    \label{consMassDiv}
    \partial_t \dens + \diver{}\mom &= 0 \, , \\
    \label{consMomDiv}
    \partial_t \mom + \diver{}(\dens^{-1} \mom \mom^\transp -
    \mu^{-1}\Bfield \Bfield^\transp + \mathbb{I} p) &=
                                                      \bzero  \, , \\
    \label{consMechEDiv}
    \partial_t \tote + \diver{}\big(\tfrac{\mom}{\dens} (\tote + p) -
    \Bfield (\Bfield^\transp \tfrac{\mom}{\dens})\big) &= 0, \\
    \label{consBfieldDiv}
    \partial_t \Bfield + \diver{} (\rho^{-1}\Bfield \mom^\transp - \rho^{-1}
    \mom \Bfield^\transp) &= \bzero,
  \end{align}
\end{subequations}
where the total energy $\tote =
\frac{1}{2 \rho}|\mom|^2 + \rho e + \tfrac{1}{2\mu} |\Bfield|^2$ includes
the contribution from the magnetic field. The regularized system reads:
\begin{subequations}\label{MHDsystemDivViscosity}
  \begin{align}
    \label{viscousConsMassDiv}
    \partial_t \dens + \diver{}\mom &= \epsilon\Delta\rho \, , \\
    \label{viscousConsMomDiv}
    \partial_t \mom + \diver{}\left(\dens^{-1} \mom \mom^\transp -
    \mu^{-1}\Bfield \Bfield^\transp
    + \mathbb{I} \left(p+\tfrac{1}{2 \mu}|\Bfield|^2\right)\right) &=
    \epsilon\Delta\mom  \, , \\
    \label{viscousConsMechEDiv}
    \partial_t \tote + \diver{}\big(\tfrac{\mom}{\dens} (\tote + p) -
    \Bfield (\Bfield^\transp \tfrac{\mom}{\dens})\big) &=
    \epsilon\Delta\left(\tote-\tfrac{1}{2\mu}|\Bfield|^2\right), \\
    \label{viscousConsBfieldDiv}
    \partial_t \Bfield + \diver{} (\rho^{-1}
    \mom \Bfield^\transp - \rho^{-1}\Bfield \mom^\transp) &= \bzero \, .
  \end{align}
\end{subequations}

Note that there is no viscous regularization in \eqref{viscousConsBfieldDiv}.
In addition, the magnetic pressure is subtracted from the total energy in the
viscous regularization term on the right hand side of
\eqref{viscousConsMechEDiv}. The difference between the viscous
regularization in reference \cite{Dao2022a} and that one in
expression \eqref{MHDsystemDivViscosity} is that the magnetic regularization was
removed. In this section, we prove that even without the magnetic regularization
terms, the state $\state = [\rho, \mom, \tote, \Bfield]^\transp$ of
\eqref{MHDsystemDivViscosity} satisfies: positivity of density, positivity of
internal energy, and minimum entropy principles, for all time. Moreover,
\eqref{MHDsystemDivViscosity} is compatible with all the generalized entropy
inequalities. These results can be obtained with slight modifications of the
proofs in \cite{Dao2022a}.

\begin{definition}[Specific entropy, Gibbs identity and
  physical restrictions]\label{CompleteEOS} Let $\inte(\state) = \tote -
  \tfrac{1}{2 \rho}|\mom|^2 - \tfrac{1}{2\mu} |\Bfield|^2$ denote the
internal energy, $\specinte(\state) = \dens^{-1} \inte(\state)$ denote
specific internal energy, and $v = \rho^{-1}$
  be the specific volume. Let $s = s(\rho, \specinte): \mathbb{R}^+ \times
  \mathbb{R}^+ \rightarrow \mathbb{R}$ denote the specific entropy. Assuming
that
  the exact differential of $s = s(\rho, \specinte)$, meaning $\mathrm{d}s =
  \frac{\partial s}{\partial e} \mathrm{d} e + \frac{\partial s}{\partial \rho}
  \mathrm{d} \rho$, is consistent with Gibbs' differential relationship $
  \mathrm{d}s = \tfrac{1}{\theta}\mathrm{d}\specinte + \frac{p}{\theta}
  \mathrm{d}v$, where $\theta$ is the temperature, implies that
  \begin{align*}
    \tfrac{\partial s}{\partial e} = \tfrac{1}{\theta}
    \, \text{,} \
    \tfrac{\partial s}{\partial \rho} = - \tfrac{p}{\theta \rho^2} \, ,
  \end{align*}
  combining both we obtain the formula for the pressure $p = \rho^2
  \frac{\partial s}{\partial \rho}[\frac{\partial s}{\partial \specinte}]^{-1}$.
  In order for $s(\rho, \specinte)$ to be physically meaningful it has to
satisfy
  some mathematical restrictions. An exhaustive list of restrictions can be
found
  in \cite{Meni1989, Guer2014}. In this manuscript we will only assume that
  $\frac{\partial s}{\partial \specinte} > 0$, implying positivity of the
  temperature, and that $-s$ is strictly convex with for any $\dens, \specinte >
  0$, see \cite[p. 3]{Dao2022a}. We will use the shorthand notation $s_e :=
  \frac{\partial s}{\partial \specinte}$ and $s_\rho := \frac{\partial
    s}{\partial\rho}$.
\end{definition}

\begin{lemma}[Positivity of density, see \cite{Guer2014,
    Dao2022a}]\label{lem:DensityPositivity}
  Assuming sufficient smoothness and boundedness of the solution, the density
  solution satisfies the following property
\begin{align*}
\mathrm{essinf}_{\xcoord \in \mathbb{R}^{d}} \rho(\xcoord, t) > 0, \quad
\forall t>0.
\end{align*}
\end{lemma}

The proof of Lemma~\ref{lem:DensityPositivity} merely depends on the mass
equation \eqref{viscousConsMassDiv}. This is a well-known result for which
detailed proof can be found in \cite{Guer2014}.

% In simple words,
% positivity of density can be guaranteed if the corresponding elliptic term
% $\epsilon\Delta\rho$ is added for any given $\epsilon > 0$.

%% Next, we introduce the following notations,
%% \[
%%   \begin{aligned}
%%     \bl & = \epsilon\nabla(\rho e), &&& \polG &= \epsilon\rho\nabla^s\vel,\\
%%     \bh & = \bl-\frac{|\vel|^2}{2}(\epsilon\Delta\rho), &&& \polg & =
%     \polG+(\epsilon\Delta\rho)\otimes\vel.\\
%%   \end{aligned}
%% \]

\begin{lemma}[Minimum principle of the specific
entropy]\label{lem:MinimumEntropyPrinciple}
Assume sufficient smoothness and that the density and the internal energy
uniformly converge to constant states outside a compact domain of interest. The
minimum entropy principle holds:
\begin{align*}
    \inf_{\xcoord\in\mathbb{R}^d}s(\rho(\xcoord, t), e(\xcoord, t)) \geq
    \inf_{\xcoord\in\mathbb{R}^d}s_0(\xcoord),
\end{align*}
where $s (\rho, e)$ is the specific entropy, see Definition \ref{CompleteEOS},
and $s_0$ is the initial specific entropy.
\end{lemma}

\begin{proof}
Multiplying \eqref{viscousConsMomDiv} with $\vel$ gives
\begin{align}
  \begin{split}\label{proof:MinimumEntropy:m}
    &\rho\left(\partial_t\left(\tfrac{1}{2}|\vel|^2\right)
      +\vel\cdot\nabla\left(\tfrac{1}{2}|\vel|^2\right)\right)
    +|\vel|^2\epsilon\Delta\rho+\vel\cdot\nabla p  \\
    & \ \ \ -\vel\cdot\diver{}\left(\mu^{-1}\Bfield \Bfield^\transp
     - \tfrac{1}{2\mu} \mathbb{I}|\Bfield|^2\right)
     -\epsilon\vel\cdot\Delta\mom = 0.
  \end{split}
\end{align}
Multiplying \eqref{viscousConsBfieldDiv} with $\Bfield$ gives
\begin{equation}\label{proof:MinimumEntropy:B}
  \rho\left(\partial_t\left(\tfrac{\rho^{-1}|\Bfield|^2}{2}\right)
  +\vel\cdot\nabla\left(\tfrac{\rho^{-1}|\Bfield|^2}{2}\right)\right)
  +\tfrac{\rho^{-1}
    |\Bfield|^2}{2}\epsilon\Delta\rho
  + \tfrac{\rho^{-1}|\Bfield|^2}{2}\rho\diver{}\vel
  - \Bfield\cdot(\Bfield\cdot\nabla)\vel = 0.
\end{equation}
Subtracting \eqref{proof:MinimumEntropy:m} and \eqref{proof:MinimumEntropy:B}
from \eqref{viscousConsMechEDiv} gives
\begin{equation}\label{proof:MinimumEntropy:e}
  \rho(\partial_t e+ \vel\cdot\nabla e)
  +\left(e- \tfrac{1}{2}|\vel|^2\right)\epsilon\Delta\rho
  + p\diver{}\vel-\epsilon\Delta
  \left(\rho e + \tfrac{1}{2\rho}|\mom|^2\right)
  + \epsilon\vel\cdot\Delta\mom = 0,
\end{equation}
which describes the evolution of the internal energy. Multiplying the mass
equation with $\rho s_{\rho}$, multiplying \eqref{proof:MinimumEntropy:e} with
$s_e$, adding them together, then applying the chain rules $\nabla s =
s_e\nabla
e + s_{\rho}\nabla\rho$ and $\partial_t s = s_e\partial_t e +
s_{\rho}\partial_t\rho$, we obtain the entropy conservation equation
\begin{equation}\label{proof:MinimumEntropy:s}
  \rho(\partial_ts + \vel\cdot\nabla s)+(e s_e-\rho
s_{\rho})\epsilon\Delta\rho
  + s_e\left(-\tfrac{1}{2}|\vel|^2\epsilon\Delta\rho-\epsilon\Delta
    \left(\rho e +
      \tfrac{1}{2\rho}|\mom|^2\right)+\epsilon\vel\cdot\Delta\mom\right)
  = 0.
\end{equation}
Let $\ell\coloneqq s_e^{-1}(e s_e - \rho s_{\rho})\epsilon\nabla\rho
+ \epsilon\rho s_e^{-1}\nabla s$. We can rewrite
\eqref{proof:MinimumEntropy:s}
as
\begin{equation}\label{proof:MinimumEntropy:s2}
  \rho(\partial_ts + \vel\cdot\nabla s)+(e s_e-\rho
s_{\rho})\epsilon\Delta\rho-s_e\diver{}\ell-s_e\epsilon\rho\nabla\vel:\nabla\vel
  = 0.
\end{equation}
Let $J\coloneqq -(\epsilon\nabla\rho)\cdot\nabla(e s_e-\rho
s_{\rho})+\ell\cdot\nabla s_e + \epsilon\nabla\rho\cdot\nabla s$. It can be
proved that $J \leq 0$, which follows from the strict convexity of $- s$,
see \cite[Lemma 3]{Dao2022a}. Therefore, we
rewrite \eqref{proof:MinimumEntropy:s2} as
\begin{align}\label{proof:MinimumEntropy:s3}
  \rho(\partial_t s + \vel\cdot\nabla s) - \diver{}(\epsilon\rho\nabla s) -
  \epsilon\nabla\rho\cdot\nabla s = -J + s_e(\epsilon\rho\nabla\vel) :
\nabla\vel,
\end{align}
where the right hand side is non-negative. In the regular case when
$\inf_{\xcoord\in\mathbb{R}^d}s(\xcoord, t)$ is reached at, let us say
$\bar{\xcoord}(t)$, inside a compact domain $\Omega \subset \mathbb{R}^d$, we
have $\nabla s(\bar{\xcoord}, t) = 0$ and $\Delta s (\bar{\xcoord}, t) \geq 0$
due to smoothness. From \eqref{proof:MinimumEntropy:s3}, we have that
$\partial_t s(\bar{\xcoord}, t) \geq 0$ since $\rho > 0$. This says that
$\inf_{\xcoord\in\mathbb{R}^d}s(\xcoord, t)$ is always increasing in time.
Therefore, we have the conclusion. However, if
$\inf_{\xcoord\in\mathbb{R}^d}s(\xcoord, t)$ is reached at
$\xcoord\rightarrow\infty$, then we have
$\inf_{\xcoord\in\mathbb{R}^d}s(\xcoord, t) = x^* \geq
\inf_{\xcoord\in\mathbb{R}^d} s_0(\xcoord)$ where
$\inf_{\xcoord\in\mathbb{R}^d}s(\xcoord, t)\rightarrow x^*$ as $\xcoord
\rightarrow \infty$ due to the uniform convergence assumption. The proof is
complete.
\end{proof}

Let $f(s)$ be a twice differentiable function with $s(\state) = s(\rho,
e(\state))$ being the specific entropy. Consider a class of strictly convex
functions $\eta(\state) = -\rho f(s(\state))$ which are known as generalized
Harten entropies. The following lemma is a direct consequence of
Lemma~\ref{lem:MinimumEntropyPrinciple}.

\begin{lemma}[Generalized entropy inequalities]\label{lem:generalized_entropies}
  Any smooth solution to \eqref{MHDsystemDivViscosity} satisfies the entropy
  inequality
  \begin{align*}
    \partial_t(\rho f(s)) + \mathrm{div}(\vel \rho f(s)
    - \epsilon\rho\nabla f(s) - \epsilon f(s)\nabla\rho) \geq 0.
  \end{align*}
\end{lemma}

\begin{proof}
  Multiplying both sides of \eqref{proof:MinimumEntropy:s3} with $f'(s)$ gives
  \begin{equation}\label{proof:EntropyInequalities:s}
    \begin{aligned}
      \rho(\partial_t f(s) + \vel\cdot\nabla f(s))&-\diver{}(\epsilon\rho\nabla
      f(s)) + \epsilon\rho f''(s)|\nabla s|^2 - \epsilon f'(s)
\nabla\rho\cdot\nabla s
      =\\& -J f'(s)
      + f'(s)s_e(\epsilon\rho\nabla\vel):\nabla\vel.
    \end{aligned}
  \end{equation}
  Multiplying the mass equation with $\rho$ and add it to
  \eqref{proof:EntropyInequalities:s}, we have
  \begin{align*}
    \partial(\rho f(s)) + \diver{}(\rho\vel f(s)) & -
                                                    \diver{}(\epsilon\rho\nabla
f(s) + \epsilon f(s)\nabla\rho) = \\
                                                  & - \epsilon\rho f''(s)|\nabla
s|^2 - J f'(s)
                                                    +
f'(s)s_e(\epsilon\rho\nabla\vel):\nabla\vel.
  \end{align*}
  By the strict convexity of $-\rho f(s)$, we can show that $- \epsilon\rho
  f''(s)|\nabla s|^2 - J f'(s) \geq 0$ and $f'(s) > 0$, see \cite[Theorem 3.4,
  3.5]{Dao2022a}. By the assumption that
  the temperature is positive, we have $s_e > 0$. Therefore, the inequality of
  the lemma always holds true.
\end{proof}

%%%%%%%%%%%%%%%%%%%%%%%%%%%%%%%%%%%%%%%%%%%%%%%%%%%%%%%%%%%%%%%%%%%%%%%%%%%%%%%%
%%%%%%%%%%%%%%%%%%%%%%%%%%%%%%%%%%%%%%%%%%%%%%%%%%%%%%%%%%%%%%%%%%%%%%%%%%%%%%%%
%%%%%%%%%%%%%%%%%%%%%%%%%%%%%%%%%%%%%%%%%%%%%%%%%%%%%%%%%%%%%%%%%%%%%%%%%%%%%%%%
%%%%%%%%%%%%%%%%%%%%%%%%%%%%%%%%%%%%%%%%%%%%%%%%%%%%%%%%%%%%%%%%%%%%%%%%%%%%%%%%
%%%%%%%%%%%%%%%%%%%%%%%%%%%%%%%%%%%%%%%%%%%%%%%%%%%%%%%%%%%%%%%%%%%%%%%%%%%%%%%%

\subsection{Energy balance of the non-divergence formulation}\label{sec:energy}

\begin{proposition}[Total energy-balance] The MHD system \eqref{MHDsystem}
  satisfies the following formal energy-flux balance:
  \begin{align}
    \label{EnergyBal}
    \partial_t \int_{\Omega} \totme + \tfrac{\mu}{2}|\Hfield|^2 \dx
    + \int_{\partial\Omega} (\totme + p) \tfrac{\mom}{\dens}\cdot \normal
    - \mu (\tfrac{\mom}{\dens} \times \Hfield)\cdot(\Hfield \times \normal)
    \ds &= 0
  \end{align}
\end{proposition}

\begin{proof} Integrating \eqref{consMechE} in space and using the divergence
  theorem we get
  \begin{align}
    \label{ProfTotMe}
    \int_{\Omega} \partial_t \totme \dx +
    \int_{\partial\Omega} (\totme + p) \tfrac{\mom}{\dens}\cdot \normal\ds
    + \mu  \int_{\Omega} (\Hfield \times  \curl{}\Hfield) \cdot
    \tfrac{\mom}{\dens} \dx &= \bzero \, .
  \end{align}
  Multiplying \eqref{consBfield} by $\mu \Hfield$, using integration by
  parts formula
  \begin{align*}
    \int_{\Omega} \curl{}\boldsymbol{u} \cdot \vel \dx= \int_{\partial\Omega}
    (\boldsymbol{u} \times \vel)\cdot \normal \ds+ \int_{\Omega} \boldsymbol{u}
\cdot
    \curl{}\vel \dx \, ,
  \end{align*}
  and reorganizing the terms we get:
  \begin{align}
    \label{ProfB}
    \partial_t \int_{\Omega} \tfrac{\mu}{2}|\Hfield|^2 \dx
    - \mu \int_{\partial\Omega} [(\tfrac{\mom}{\dens} \times \Hfield)\times
    \Hfield]\cdot \normal \ds
    - \mu \int_{\Omega} (\tfrac{\mom}{\dens} \times \Hfield)\cdot
    \curl{}\Hfield \dx  = 0 \, .
  \end{align}
  Using the property $
  \curl{}\Hfield \cdot (\tfrac{\mom}{\dens} \times \Hfield) =
  \tfrac{\mom}{\dens} \cdot (\Hfield \times \curl{}\Hfield ) \, ,
  $ and inserting this identity into \eqref{ProfB} yields
  \begin{align}\label{ReplcementRhs}
    \partial_t \int_{\Omega} \tfrac{\mu}{2}|\Hfield|^2 \dx
    - \mu \int_{\Omega} (\Hfield \times \curl{}\Hfield )\cdot
    \tfrac{\mom}{\dens} \dx
    - \mu \int_{\partial\Omega} [(\tfrac{\mom}{\dens} \times \Hfield)\times
    \Hfield]\cdot \normal \ds = 0.
  \end{align}
  Finally, adding \eqref{ReplcementRhs} to \eqref{ProfTotMe}, and using
  properties of the triple product yields the desired result.
\end{proof}

\begin{remark}[Energy conservation and boundary conditions]\label{BundaryIso}
  As noted in the introduction, non-divergence formulation \eqref{MHDsystem} and
  divergence formulation \eqref{MHDsystemDiv} should be treated as
  different models. As such, each formulation accommodates
  complementary set of boundary conditions. We start by noting that the
  total energy considered in \eqref{EnergyBal} is $\int_{\Omega} \totme +
  \tfrac{\mu}{2}|\Hfield|^2 \dx$ is not a linear functional but rather the sum
  of a linear and a quadratic functional. As usual, conservation of total
  energy $\int_{\Omega} \totme + \tfrac{\mu}{2}|\Hfield|^2 \dx$ holds for
  the trivial case of periodic boundary conditions. Inspecting identity
  \eqref{EnergyBal}, we note that another simple scenario where total energy is
  conserved is when $\mom\cdot\normal\equiv 0$ and $\Hfield \times \normal
  \equiv \bzero$ on the entirety of the boundary. Tangent boundary conditions,
  such as $\Hfield \times \normal \equiv\bzero$, can be enforced in the
  curl-conforming framework such as the Sobolev space $H(\curl{})$, see for
  instance \cite{Monk2003, Ciarlet2018}, which will be used to discretize
  $\Hfield$. Note that the normal boundary conditions $(\mu \Hfield)\cdot
  \normal = 0$, traditionally used for the divergence formulation
  \eqref{MHDsystemDiv}, are not meaningful in the context of non-divergence
  model \eqref{MHDsystem}.
\end{remark}

\begin{remark}[Simplifying assumptions] In the remainder of the paper, in order
  to simplify arguments in relationship to boundary conditions, we assume that
  periodic boundary conditions are used, or that the initial data is compactly
  supported and that the final time is small enough to prevent waves from
reaching
  the boundary. Alternatively, we could assume that boundary conditions
  $\mom\cdot\normal \equiv 0$ and $\Hfield \times \normal \equiv \bzero$ are
used
  on the entirety of the boundary.
\end{remark}

%%%%%%%%%%%%%%%%%%%%%%%%%%%%%%%%%%%%%%%%%%%%%%%%%%%%%%%%%%%%%%%%%%%%%%%%%%%%%%%%
%%%%%%%%%%%%%%%%%%%%%%%%%%%%%%%%%%%%%%%%%%%%%%%%%%%%%%%%%%%%%%%%%%%%%%%%%%%%%%%%
%%%%%%%%%%%%%%%%%%%%%%%%%%%%%%%%%%%%%%%%%%%%%%%%%%%%%%%%%%%%%%%%%%%%%%%%%%%%%%%%
%%%%%%%%%%%%%%%%%%%%%%%%%%%%%%%%%%%%%%%%%%%%%%%%%%%%%%%%%%%%%%%%%%%%%%%%%%%%%%%%

\subsection{Euler's equation with forces}
Consider Euler's system subject to the effect of a force, that is
\begin{align}
  \label{EulerWithSources}
  \tfrac{\partial}{\partial t}\state + \diver{}
  \flux(\state) = \bv{s}(\force),
\end{align}
with
\begin{align}\label{EulerForceSystem}
  \state = \begin{bmatrix}
    \rho \\
    \mom \\
    \totme
  \end{bmatrix} \ , \ \
  \flux(\state) = \begin{bmatrix}
    \mom^\transp \\
    \dens^{-1} \mom \mom^\transp + \mathbb{I} p \\
    \dens^{-1 } \mom^\transp (\totme + p)
  \end{bmatrix} \ , \ \
  \bv{s}(\force) =
  \begin{bmatrix}
    0 \\ \force \\ \dens^{-1}\mom \cdot \force
  \end{bmatrix}.
\end{align}
In particular, system \eqref{consMass}-\eqref{consMechE} can be rewritten as in
\eqref{EulerWithSources}-\eqref{EulerForceSystem} using the particular choice of
force $\force = - \mu  \Hfield \times \curl{}\Hfield$. For
any force $\force$ we have the property described in the following lemma.

\begin{lemma}[Invariance of entropy-like functionals
\cite{TomasMaier2022}]\label{lem:InteInvar} Let
  $\state = [\dens, \mom, \totme] \in \mathbb{R}^{d+2}$ be the state of Euler's
  system. Let
  $\Psi(\state):\mathbb{R}^{d+2} \rightarrow
  \mathbb{R}$ be \underline{any} arbitrary functional of the state satisfying
the
  functional dependence $\Psi(\state):= \psi(\dens, \inte (\state))$, where
  $\inte (\state) := \totme - \frac{|\mom|^2}{2 \dens}$ is the internal
  energy per unit volume. Then we have that
  \begin{align}
    \label{resultinv}
    \nabla_{\state} \Psi(\state) \cdot \bv{s}(\force) \equiv 0 ,
  \end{align}
  where $\nabla_{\state}$ is the gradient with respect to the state, \ie
  $\nabla_{\state} = \big[\frac{\partial}{\partial\dens},
  \frac{\partial}{\partial\mom_1}, ...,
  \frac{\partial}{\partial\mom_d},
  \frac{\partial}{\partial\totme}\big]^\transp$.
\end{lemma}

\begin{proof}
  Using the chain rule we observe that $\nabla_{\state} \Psi(\state) =
  \frac{\partial\psi}{\partial\dens} \nabla_{\state}\dens +
  \frac{\partial\psi}{\partial\inte} \nabla_{\state}\inte$, where
  \begin{align*}
    \nabla_{\state}\dens &= [1,0, ..., 0]^\transp \in \mathbb{R}^{d+2},
    \\
    \nabla_{\state}\inte &=
\big[\tfrac{|\mom|^2}{\dens^2},-\tfrac{\mom_1}{\dens}, ...,
                           -\tfrac{\mom_d}{\dens}, 1\big]^\transp \in
\mathbb{R}^{d+2}.
  \end{align*}

  Taking the product with $\bv{s}(\force)$ we get
  \begin{align*}
    \nabla_{\state} \Psi(\state) \cdot \bv{s}(\force)
    &= \frac{\partial\psi}{\partial\dens} \underbrace{\nabla_{\state}\dens
      \cdot \bv{s}(\force)}_{=\, 0} + \frac{\partial\psi}{\partial\inte}
      \nabla_{\state}\inte \cdot \bv{s}(\force)
    \\
    &= \frac{\partial\psi}{\partial\inte} (- \dens^{-1} \mom \cdot \force +
      \dens^{-1} \mom \cdot \force) = 0.
  \end{align*}
\end{proof}

\begin{remark}[Colloquial interpretation]\label{Colloq}
  Lemma \ref{lem:InteInvar} is simply saying that the evolution in time
  of an arbitrary functional $\Psi(\state)$ satisfying the functional dependence
  $\Psi(\state):= \psi(\dens, \inte (\state))$ is independent of the force
  $\force$. This follows directly by taking the dot-product of
  \eqref{EulerWithSources} with $\nabla_{\state}
  \Psi(\state)$ to get
  \begin{align*}
    \nabla_{\state}\Psi(\state)\cdot \tfrac{\partial}{\partial t}\state
    \;=\;
    \tfrac{\partial}{\partial t}\Psi(\state)
    \;=\;
    - \nabla_{\state} \Psi(\state) \cdot \diver{} \flux(\state) +
    \underbrace{\nabla_{\state} \Psi(\state) \cdot
    \bv{s}(\force)}_{\equiv\,0} .
  \end{align*}
  In particular, this holds true when $\Psi(\state):= \inte (\state)$.
  Similarly, we can apply Lemma \ref{lem:InteInvar} to the specific internal
  energy $\specinte(\state) = \dens^{-1} \inte(\state)$ since
  $\specinte(\state)$ satisfies the functional dependence $\specinte(\state) =
  \psi(\dens, \inte  (\state))$ as well. We also note that condition
  \eqref{resultinv} is related to the so-called ``complementary
  degeneracy requirements'' usually invoked in GENERIC systems, see
  \cite{Ottin1997}.
\end{remark}

%%%%%%%%%%%%%%%%%%%%%%%%%%%%%%%%%%%%%%%%%%%%%%%%%%%%%%%%%%%%%%%%%%%%%%%%%%%%%%%%
%%%%%%%%%%%%%%%%%%%%%%%%%%%%%%%%%%%%%%%%%%%%%%%%%%%%%%%%%%%%%%%%%%%%%%%%%%%%%%%%
%%%%%%%%%%%%%%%%%%%%%%%%%%%%%%%%%%%%%%%%%%%%%%%%%%%%%%%%%%%%%%%%%%%%%%%%%%%%%%%%
%%%%%%%%%%%%%%%%%%%%%%%%%%%%%%%%%%%%%%%%%%%%%%%%%%%%%%%%%%%%%%%%%%%%%%%%%%%%%%%%

% \newpage
\subsection{Splitting of the differential operator}

\begin{remark}[Choice of splitting] We consider the splitting for the system
  \eqref{MHDsystem} in two evolutionary operators:
\begin{gather}\label{Operator1}
\text{Operator \#1 }
\left\{
  \begin{split}
    \partial_t \dens + \diver{}\mom &= 0 \, , \\
    \partial_t \mom
    + \diver{}(\dens^{-1} \mom \mom^\transp + \mathbb{I} p) &=  \bzero \, , \\
    \partial_t \totme
    + \diver{}\big(\tfrac{\mom}{\dens} (\totme + p) \big) &= 0 \,
    , \\
    \partial_t \Hfield &= \bzero \, ,
  \end{split}\right.
\end{gather}
  and
  \begin{gather}\label{Operator2}
    \text{Operator \#2 }
    \left\{
      \begin{split}
        \partial_t \dens &= 0 \, , \\
        \partial_t \mom &=
        - \mu  \Hfield \times  \curl{}\Hfield  \, , \\
        \partial_t \totme &=
        - \mu  (\Hfield \times  \curl{}\Hfield) \cdot \tfrac{\mom}{\dens} \, ,
\\
        \partial_t \Hfield &= \curl{} (\vel \times \Hfield) \, .
      \end{split}\right.
  \end{gather}
\end{remark}

Given some initial data $\state^n = [\rho^n,\mom^n, \totme^n, \Hfield^n]$ for
each one of these operators, we would like to know what properties are
preserved by their evolution.

\begin{proposition}[Evolution of Operator \#1: preservation of linear
  invariants and pointwise stability properties] Assume periodic boundary
  conditions. Assume that the initial data at time $t_n$ is admissible, meaning
  $\state^n(\xcoord) = [\rho^n, \mom^n, \totme^n](\xcoord)\in \mathcal{A}$ for
  all $\xcoord \in \Omega$ with $\mathcal{A}$ defined as
  \begin{align}\label{DefAdmissibleA}
    \mathcal{A} = \big\{ [\rho, \mom, \totme]^\transp \in
    \mathbb{R}^{d+2}\, \big| \, \rho > 0 , \, \totme - \tfrac{1}{2}
    \tfrac{|\mom|^2}{\rho} > 0 \big\} \, .
  \end{align}
  Then, the evolution of Operator \#1 from time $t_n$ to $t_{n+1}$,
  preserves the following linear invariants
  \begin{align}
    \begin{split}\label{EulerConservation}
      \int_{\Omega} \dens^{n+1} \dx = \int_{\Omega} \dens^n \dx \ , \ \
      \int_{\Omega} \mom^{n+1} \dx = \int_{\Omega} \mom^n \dx \ , \\
      \int_{\Omega} \totme^{n+1} \dx = \int_{\Omega} \totme^n \dx \ , \ \
      \int_{\Omega} \Hfield^{n+1} \dx = \int_{\Omega} \Hfield^n \dx \ , \ \
    \end{split}
  \end{align}
  as well as the pointwise stability properties
  \begin{align}\label{EulerPointwise}
    \dens^{n+1}(\xcoord) \geq 0 \ , \ \ 
    s(\state^{n+1}(\xcoord)) \geq \min_{\xcoord \in \Omega} s(\state^n(\xcoord))
\ ,  \
    \varepsilon(\state^{n+1}(\xcoord))\geq 0 .
  \end{align}
  for all $\xcoord \in \Omega$.
\end{proposition}
Note that $\int_{\Omega} \Hfield^{n+1} \dx =
  \int_{\Omega} \Hfield^n \dx $ follows trivially from the fact that $\partial_t
  \Hfield \equiv \bzero$ for the case of Operator \#1.
Properties \eqref{EulerConservation} are a consequence of the divergence
theorem. On the other hand, establishing properties \eqref{EulerPointwise} is
rather technical, the reader is referred to \cite{Tadmor1986, Guer2014}.
We note in passing that positivity of the internal energy is not
a direct property, but rather a consequence of the positivity of density and
minimum principle of the specific entropy.

\begin{corollary}[Energy-stability of Operator \#1: Linear + Quadratic
  functional]\label{CorollaryEnergyOp1} Assume periodic boundary conditions,
  then the evolution described
  by Operator \#1 satisfies the following energy-estimate
  \begin{align*}
    \int_{\Omega}\totme^{n+1} + \tfrac{\mu}{2} |\Hfield^{n+1}|^2 \dx
    =
    \int_{\Omega}\totme^{n} + \tfrac{\mu}{2} |\Hfield^{n}|^2 \dx
  \end{align*}
\end{corollary}

\begin{proof} This follows from the conservation property $\int_{\Omega}
  \totme^{n+1} \dx = \int_{\Omega} \totme^n \dx$, then we add $\int_{\Omega}
  \tfrac{\mu}{2 }|\Hfield^n|^2 \dx $ to both sides of the equality, and use the
  fact that $\Hfield^{n+1} \equiv \Hfield^n$ since $\partial_t \Hfield \equiv
  \bzero$.
\end{proof}

Regarding Operator \#2, we start by noting that since $\partial_t \rho
\equiv 0$, we can rewrite \eqref{Operator2} as
\begin{gather}
  \label{Operator2rev}
  \text{Operator \#2 }
  \left\{
    \begin{split}
      \partial_t\rho &= 0 \\
      \rho \partial_t\vel &=
      - \mu  \Hfield \times  \curl{}\Hfield  \, , \\
      \partial_t \totme &=
      - \mu  (\Hfield \times  \curl{}\Hfield) \cdot \vel \, , \\
      \partial_t \Hfield &= \curl{} (\vel \times \Hfield) \, ,
    \end{split}\right.
\end{gather}
and note that only the evolution of $\vel$ and $\Hfield$ are actually
coupled. Assume periodic boundary conditions, multiply the evolution
equation for $\vel$ with a smooth test function $\veltest$ and the evolution
equation for $\Hfield$ with a vector-valued smooth test function $\Htest$
integrate by parts, then we get:
\begin{align}
  \label{SourceSmoothTest}
  \begin{split}
    (\rho \partial_t\vel,\veltest) &= - \mu
    (\Hfield \times \curl{}\Hfield,\veltest) \, , \\
    (\partial_t \Hfield,\Htest) &= (\Hfield \times \curl{}\Htest, \vel) \, ,
  \end{split}
\end{align}
We will discretize \eqref{SourceSmoothTest} in space and time, see Section
\ref{sec:SourceUpdate}. It is clear that, in order to make sense of the
integration by parts used to derive \eqref{SourceSmoothTest} the weak or
distributional $\curl{}$ of $\Hfield$ should be well defined. Therefore, it is
natural to consider a curl-conforming space discretization for $\Hfield$. Note
that tangent boundary conditions $\Hfield \times \normal \equiv \bzero$, which
can be directly enforced in the curl-conforming framework, are useful to achieve
energy-isolation of the MHD system, see Remark \ref{BundaryIso}.

% We will make no attempt specify the functional framework required to guarantee
% the existence and uniqueness of solutions of \eqref{SourceSmoothTest} in the
% infinite dimensional setting since this is a largely open issue.

\begin{proposition}[Evolution of Operator \#2: global energy stability and
  pointwise invariants]\label{PropStwoMore} Let $\state^n = [\rho^n,
  \mom^n, \totme^n, \Hfield^n]^\transp$ be the initial data. Assume periodic
  boundary conditions, then the evolution of
  described by Operator \#2 as defined in \eqref{Operator2rev}, preserves the
  following global quadratic-invariant:
  \begin{align}\label{GlobalBalSource}
    \int_{\Omega} \tfrac{1}{2} \dens^{n+1} |\vel^{n+1}|^2
    + \tfrac{\mu}{2} |\Hfield^{n+1}|^2 \dx =
    \int_{\Omega} \tfrac{1}{2} \dens^n|\vel^{n}|^2
    + \tfrac{\mu}{2} |\Hfield^{n}|^2 \dx \, ,
  \end{align}
  a well as pointwise invariance of the internal energy
  \begin{align}
    \label{InvarInte}
    (\totme^{n+1} - \tfrac{1}{2} \dens^{n+1} |\vel^{n+1}|^2)(\xcoord)
    =
    (\totme^{n} - \tfrac{1}{2} \dens^{n} |\vel^{n}|^2)(\xcoord)
  \end{align}
  for all $\xcoord \in \Omega$, with $\dens^{n+1}(\xcoord) = \dens^{n}(\xcoord)$
  since density does not evolve for the case of Operator \#2.
\end{proposition}

\begin{proof} Energy stability \eqref{GlobalBalSource} follows by taking
  $\veltest = \vel$ and $\Htest = \mu \Hfield$ in \eqref{SourceSmoothTest} and
  adding both lines. On the other hand, the invariance of the internal energy
  \eqref{InvarInte} is a direct consequence \eqref{resultinv} and Remark
  \ref{Colloq}.
\end{proof}

\begin{corollary}[Energy-stability of Operator \#2: Linear + Quadratic
  functional]\label{cor:totalEnergy} Under the assumptions of
  Proposition \ref{PropStwoMore}, the evolution of Operator \#2 satisfies the
  following energy-balance
  \begin{align}
    \label{TotEstab}
    \int_{\Omega}\totme^{n+1} + \tfrac{\mu}{2} |\Hfield^{n+1}|^2 \dx
    =
    \int_{\Omega}\totme^{n} + \tfrac{\mu}{2} |\Hfield^{n}|^2 \dx
  \end{align}
\end{corollary}

\begin{proof} Integrating \eqref{InvarInte} with respect to space we get
  \begin{align}
    \label{GlobalInte}
    \int_{\Omega} \totme^{n+1} - \tfrac{1}{2} \dens^{n+1} |\vel^{n+1}|^2 \dx
    =
    \int_{\Omega} \totme^{n} - \tfrac{1}{2} \dens^{n} |\vel^{n}|^2 \dx .
  \end{align}
  Identity \eqref{TotEstab} follows by adding \eqref{GlobalBalSource} and
  \eqref{GlobalInte} leading to the cancellation of kinetic energy terms.
\end{proof}

\begin{remark}[Invariant domain preservation for the evolution described by
  Operator \#2] We note that the evolution described Operator \#2 is such that
  neither density nor internal energy evolve, then: specific entropy and
  mathematical entropy remain invariant. Or equivalently in term of formulas
  \begin{align}\label{EvoConst}
    \tfrac{\partial \dens}{\partial t}  \equiv 0
    \ \ \text{and} \ \
    \tfrac{\partial}{\partial t} \big( \totme - \tfrac{1}{2 \dens}
|\mom|^2\big)
    \equiv 0
    \ \ \text{imply that} \ \
    \tfrac{\partial s}{\partial t}  \equiv 0
    \ \ \text{and} \ \
    \tfrac{\partial \eta}{\partial t}  \equiv 0.
  \end{align}
  here $\eta(\state) = -\rho s(\state)$ is the mathematical entropy. In
  conclusion: the evolution of Operator \#2 cannot meaningfully affect the
  evolution of density, internal energy, or specific entropy. Therefore,
  Operator \#2 cannot affect the preservation of invariant set properties.
  Since the mathematical entropy remains constant during the evolution
  described by Operator \#2 we also have the global estimate:
  \begin{align*}
    \int_{\Omega} \eta(\state(\xcoord, t_{n+1})) \dx
    =
    \int_{\Omega} \eta(\state(\xcoord,t_n)) \dx
    \, .
  \end{align*}
\end{remark}

\begin{remark}[Discrete-time evolution of total mechanical energy for
Operator \#2] From \eqref{Operator2rev} we note that the evolution of velocity
$\vel$ and magnetic field $\Hfield$ are independent from the evolution of the
total mechanical energy $\totme$. Therefore, in the time-discrete setting, given
an initial data $[\rho^{n},\vel^{n}, \totme^n, \Hfield^n](\xcoord)$ at time
$t_n$, we can compute $\vel^{n+1}$ and $\Hfield^{n+1}$ by integrating
\eqref{SourceSmoothTest} in time neglecting the evolution law for the total
mechanical energy $\partial_t \totme = - \mu  (\Hfield \times
\curl{}\Hfield) \cdot \tfrac{\mom}{\dens}$. Once $\vel^{n+1}$ and
$\Hfield^{n+1}$ are available, the constraint \eqref{InvarInte} identifies a
unique function $\totme^{n+1}(\xcoord)$. More precisely, we may rewrite
\eqref{InvarInte} as
\begin{align}\label{TotmeEvoRule}
  \totme^{n+1}(\xcoord) := \big( \totme^{n} - \tfrac{1}{2} \dens^n
  |\vel^{n}|^2 + \tfrac{1}{2} \dens^{n+1}|\vel^{n+1}|^2 \big)(\xcoord),
\end{align}
and use it in order to compute $\totme^{n+1}(\xcoord)$ using the data
$\dens^n$, $\dens^{n+1} = \dens^{n}$, $\totme^{n}$, $\vel^{n}$ and
$\vel^{n+1}$. This means that there is no particular use for $\partial_t \totme
= - \mu (\Hfield \times \curl{}\Hfield) \cdot \tfrac{\mom}{\dens}$. Therefore,
we use \eqref{TotmeEvoRule} in order to update total mechanical energy in the
time-discrete setting. Note that, by construction, \eqref{TotmeEvoRule}
guarantees exact preservation of the internal energy, specific internal energy,
specific entropy and mathematical entropy as well (provided that $\rho^{n+1}
\equiv \rho^n$).
% and the evolution law for the
% total mechanical energy $\partial_t \totme = - \mu  (\Hfield \times
% \curl{}\Hfield)
% \cdot \tfrac{\mom}{\dens}$ (presented in \eqref{Operator2rev}) .
% \Question{MN: Please explain it to me during meeting.}
% We have to choose either and ignore the other one.
%
% In this paper we pick
% $\tfrac{\partial}{\partial t} \big( \totme - \tfrac{1}{2 \dens}
% |\mom|^2\big)$,
\end{remark}

\begin{remark}[Induction equation as an independent object of
study]\label{InductionAdvective} We note that the induction equation, either
\eqref{consBfield} or \eqref{consBfieldDiv}, is a very interesting object in
its own right. Broadly speaking, the devise of schemes for advective PDEs
endowed with involution-like constraints is a challenging task that has
received significant attention in the last years \cite{Torr2003, Torr2004,
Besse2005, Mishra2010, Mishra2012, Koley2012, Heumann2012, Heumann2013,
Heumann2016, Sarkar2019}. However, unless we make very strong assumptions
about
the velocity field, the induction equation does not satisfy
any global stability property (\eg $L^2$-stability). Similarly, to the best of
the
authors' knowledge the induction does not satisfy any pointwise stability
property, such as max/min principles or invariant set properties. Since the
induction equation does not have natural notions of global or pointwise
stability, numerical stability not a well-defined concept. On the other
hand,
system \eqref{Operator2rev} satisfies the global stability property
\eqref{GlobalBalSource}, and the pointwise stability property
\eqref{InvarInte},
outlining quite clearly the properties numerical methods should preserve.
\end{remark}

%%%%%%%%%%%%%%%%%%%%%%%%%%%%%%%%%%%%%%%%%%%%%%%%%%%%%%%%%%%%%%%%%%%%%%%%%%%%%%%%

%%%%%%%%%%%%%%%%%%%%%%%%%%%%%%%%%%%%%%%%%%%%%%%%%%%%%%%%%%%%%%%%%%%%%%%%%%%%%%%%

%%%%%%%%%%%%%%%%%%%%%%%%%%%%%%%%%%%%%%%%%%%%%%%%%%%%%%%%%%%%%%%%%%%%%%%%%%%%%%%%

%%%%%%%%%%%%%%%%%%%%%%%%%%%%%%%%%%%%%%%%%%%%%%%%%%%%%%%%%%%%%%%%%%%%%%%%%%%%%%%%

%%%%%%%%%%%%%%%%%%%%%%%%%%%%%%%%%%%%%%%%%%%%%%%%%%%%%%%%%%%%%%%%%%%%%%%%%%%%%%%%

\section{Space and time discretization of the MHD system}\label{sec:discrete}

\subsection{Space discretization}\label{sec:SpaceDisc}

In this subsection we outline the space discretization used for Euler's
components $\{\rho, \mom, \totme\}$ and the magnetic field $\Hfield$. The ideas
advanced in this manuscript work in two space dimensions ($d=2$) as well as
three space dimensions ($d = 3$). Similarly, the scheme has no limitation on the
choice of polynomial degree and formal accuracy in space. However, for the sake
of concreteness, we focus on the case of $d = 2$ and spatial discretizations
capable of delivering second-order accuracy. In Remark \ref{FiniteElementsQuads}
we also provide the proper generalization for the case of
quadrilateral/tetrahedral meshes.

We consider a simplicial mesh $\triangulation$ and a corresponding
scalar-valued continuous finite element space $\FESpaceHypComp$ for each
component of Euler's system:
\begin{align}\label{VdefSpace}
\FESpaceHypComp &=
\big \{ v_h(\xcoord) \in \mathcal{C}^{0}(\Omega) \;\big|\;
v_h (\locglobmap_\element(\widehat{\xcoord})) \in
\mathbb{P}^1(\widehat{\element}) \;\forall \element \in \triangulation \big\}.
\end{align}
Here, $\locglobmap_\element(\widehat{\xcoord}):\widehat{\element}\to\element$
denotes a diffeomorphism mapping from the unit simplex $\widehat{\element}$ to
the physical element $\element \in \triangulation$,
and $\mathbb{P}^1(\widehat{K})$ is polynomial space of at most first degree on
the reference element. We
define $\HypVertices=\big\{1:\text{dim}(\FESpaceHypComp)\big\}$ as the
index-set of global, scalar-valued degrees of freedom corresponding to
$\FESpaceHypComp$. Similarly, we introduce the set of global shape functions
$\{\phi_{i}(\xcoord)\}_{i \in \HypVertices}$ and the set of collocation
points $\{\xcoord_{i}\}_{i \in \HypVertices}$ satisfying the property
$\phi_{i}(\xcoord_j) = \delta_{ij}$ for all $i,j \in \HypVertices$. We
assume that the partition of unity property $\sum_{i \in \HypVertices}
\phi_{i}(\xcoord) = 1$ for all $\xcoord \in \Omega$ holds true. We introduce a
number of matrices that will be used for the algebraic discretization. We
define the consistent mass matrix entries $m_{ij} \in \mathbb{R}$, lumped mass
matrix $m_i \in \mathbb{R}$, and the discrete divergence-matrix entries
$\bv{c}_{ij} \in \mathbb{R}^d$:
\begin{align}
  \label{schemeMatrices}
  m_{ij} = \int_{\Omega} \HypBasisComp_i \HypBasisComp_j\dx \ , \ \
  m_i = \int_{\Omega} \HypBasisComp_i \dx \ , \ \
  \bv{c}_{ij} = \int_{\Omega} \nabla\phi_j \phi_i \dx \, .
\end{align}
Note that the definition of $m_{ij}$ and the partition of unity property
$\sum_{i \in \HypVertices} \phi_{i}(\xcoord) = 1$ imply that $\sum_{j \in
  \HypVertices} m_{ij} = m_i$. Given two scalar-valued finite element
functions $u_h = \sum_{i \in \HypVertices} u_i \HypBasisComp_i \in
\FESpaceHypComp$ and $v_h = \sum_{i \in \HypVertices} v_i
\HypBasisComp_i\in \FESpaceHypComp$ we define the lumped inner product as
\begin{align}\label{LumpedInner}
  \langle u_h, v_h \rangle = \sum_{i \in \HypVertices} m_i u_i v_i \, .
\end{align}
For the case of vector valued functions $\boldsymbol{u}_h = \sum_{i \in
  \HypVertices} \boldsymbol{u}_i \HypBasisComp_i \in
[\FESpaceHypComp]^2$ and $\boldsymbol{v}_h = \sum_{i \in \HypVertices}
\boldsymbol{v}_i \HypBasisComp_i\in [\FESpaceHypComp]^2$ the lumped
inner-product is defined as $\langle \boldsymbol{u}_h, \boldsymbol{v}_h \rangle
= \sum_{i \in \HypVertices} m_i \boldsymbol{u}_i \cdot \boldsymbol{v}_i$.

We define the finite dimensional space
\begin{align}\label{BDMspace}
\FESpaceH = \big\{ \Htest_h \in H(\text{curl}, \Omega) \, \big | \,
[\nabla_{\widehat{\xcoord}}\locglobmap_{\element}(\widehat\xcoord)]^{\transp}
\Htest_h(\locglobmap_{\element}(\widehat{\xcoord})) \in
[\mathbb{P}^1(\widehat{\element})]^2
\  \forall \element \in \triangulation \big\}
\end{align}
which will be used to discretize the magnetic field $\Hfield$. The finite
element space $\FESpaceH$ is known as the ``rotated'' or curl-conforming
$\text{BDM}_1$ space. The primary motivation to use this space is that it the
simplest curl-conforming finite element that spans all the vector-valued
polynomial space $[\mathbb{P}_1]^2$ , therefore full second-order accuracy
should be expected in the $L^p$-norms when using this element.

Finally, we define the space
\begin{align}\label{PotSpace}
  \FESpacePot = \big\{ \EpotTest_h \in \mathcal{C}^0(\Omega) \, \big | \,
  \EpotTest_h(\locglobmap_{\element}(\widehat{\xcoord})) \in
  \mathbb{P}_2(\widehat{\element})
  \  \forall \element \in \triangulation \big\} \, .
\end{align}
It is easy to prove that the space $\FESpacePot$ satisfies the inclusion
$\nabla\FESpacePot \subset \FESpaceH$, more precisely, these two
spaces are part of a discrete exact sequence, see \cite{Arnold2014}. The space
$\FESpacePot$ is used to define the weak divergence-free property, see
Proposition \ref{WeakDivProp}.

\begin{remark}[Quadrilateral and tetrahedral meshes]\label{FiniteElementsQuads}
In the context of tensor product elements, we have to use different definitions
for the set of spaces $\FESpaceHypComp$, $\FESpaceH$ and $\FESpacePot$ defined
in \eqref{VdefSpace}, \eqref{BDMspace} and \eqref{PotSpace} respectively.
The simplest space we can use in order to discretize the components of
Euler's system is:
\begin{align}
\label{QuadMag}
\FESpaceHypComp &= \big \{ v_h(\xcoord) \in
\mathcal{C}^{0}(\Omega) \;\big|\;
v_h (\locglobmap_\element(\widehat{\xcoord})) \in
\mathbb{Q}^k(\widehat{\element}) \;\forall \element \in
\triangulation
\big\}.
\end{align}
for $k \geq 1$. On the other hand, the natural candiates for the spaces
$\FESpaceH$ and $\FESpacePot$ are
\begin{align}
\label{QuadHspace}
\FESpaceH &= \big\{ \Htest_h \in H(\text{curl}, \Omega) \, \big | \,
[\nabla_{\widehat{\xcoord}}\locglobmap_{\element}(\widehat\xcoord)]^{\transp}
\Htest_h(\locglobmap_{\element}(\widehat{\xcoord})) \in
\mathcal{N}_k(\widehat{\element})
\  \forall \element \in \triangulation \big\} \\
\label{QuadPotspace}
\FESpacePot &= \big\{ \EpotTest_h \in \mathcal{C}^0(\Omega) \, \big | \,
\EpotTest_h(\locglobmap_{\element}(\widehat{\xcoord})) \in
\mathbb{Q}^{k+1}(\widehat{\element})
\  \forall \element \in \triangulation \big\} \, .
\end{align}
with $\mathcal{N}_k(\widehat{\element}) =
[\mathbb{P}_{k,k+1,k+1}(\widehat{\element}),
\mathbb{P}_{k+1,k,k+1}(\widehat{\element}),
\mathbb{P}_{k+1,k+1,k}(\widehat{\element})]$, where
$\mathbb{P}_{p,q,r}$ denotes the space of scalar-valued polynomials of at most
$p$-th degree in the $x$-variable, $q$-th degree in the $y$-variable and
$r$-th degree in the $z$-variable. The vector-valued polynomial family
$\mathcal{N}_k(\widehat{\element})$ is the celebrated Nedelec space of the
first kind, see \cite{Brezzi2013}. The choice of spaces described in
\eqref{QuadMag}-\eqref{QuadPotspace} can be generalized straightforwardly to
arbitrary polynomial degree for both the two and three space dimensions. An
alternative to the choices described in
\eqref{QuadHspace}-\eqref{QuadPotspace}, for the specific case of
two-space dimensions and a target of second-order accuracy, is using the
$\text{BDM}_1$ space on quadrilaterals for $\FESpaceH$, also denoted as
$\mathcal{S}_1 \Lambda^1$ in the context of finite element exterior calculus,
and the serendipity element $\mathcal{S}_1 \Lambda^0$ for $\FESpacePot$,
see \cite{Arnold2014}. However, the implementation of elements from the
$\mathcal{S}_k \Lambda^r$ family, and their generalization to higher-order
polynomial degrees is slightly more technical.
\end{remark}

%%%%%%%%%%%%%%%%%%%%%%%%%%%%%%%%%%%%%%%%%%%%%%%%%%%%%%%%%%%%%%%%%%%%%%%%%%%%%%%%
%%%%%%%%%%%%%%%%%%%%%%%%%%%%%%%%%%%%%%%%%%%%%%%%%%%%%%%%%%%%%%%%%%%%%%%%%%%%%%%%
%%%%%%%%%%%%%%%%%%%%%%%%%%%%%%%%%%%%%%%%%%%%%%%%%%%%%%%%%%%%%%%%%%%%%%%%%%%%%%%%

\subsection{Discretization of the Operator \#1: minimal
  assumptions}\label{sec:hypminimal}

The central ideas advanced in this paper are compatible with most of the
existing numerical methods used to solve Euler's equation of gas dynamics.
In this subsection we limit ourselves to outline the minimal assumptions
made about the numerical scheme used to approximate the solution of Operator
\#1. We assume that, given some initial data $\state_h = [
\rho_h^n, \mom_h^n, \totme_h^n]^\transp$, a numerical approximation to the
solutions of $\state(\xcoord, t) = [\rho(\xcoord, t), \mom(\xcoord, t),
\totme(\xcoord, t)]^\transp$ at time $t_n$, we have at hand a numerical
procedure to compute the updated state as
\begin{align}\label{EulerAbstractScheme}
\{\rho_h^{n+1}, \mom_h^{n+1}, \totme_h^{n+1}, \dt_n\}
&:=
\texttt{euler\_system\_update}(\{\rho_h^n, \mom_h^{n}, \totme_h^n\}),
\end{align}
where $\{\rho_h^{n+1}, \mom_h^{n+1}, \totme_h^{n+1}\}$ is the approximate
solution at time $t_n + \dt_n$. Note that as described
in \eqref{EulerAbstractScheme}, $\dt_n$ is a return argument of the procedure
$\texttt{euler\_system\_update}$. In other words,
$\texttt{euler\_system\_update}$ determines the time-step size on its own. We
may at times, need to prescribe the time-step size used by
$\texttt{euler\_system\_update}$, in such case the interface of the method
might look as:
\begin{align*}
  \{\rho_h^{n+1}, \mom_h^{n+1}, \totme_h^{n+1}\}
  &:=
\texttt{euler\_system\_update}(\{\rho_h^n, \mom_h^{n}, \totme_h^n, \dt_n\}),
\end{align*}
where $\dt_n$ is supplied to $\texttt{euler\_system\_update}$. The
internals of $\texttt{euler\_system\_update}$ are not of much relevance.
However, we may assume that $\texttt{euler\_system\_update}$ is formally
second-order accurate, and most importantly, the following structural
properties:
\begin{itemize}
\item[\itemizebullet] \textit{Collocated discretization.} We assume that all the
  components of Euler's system \eqref{Operator1} are discretized in a collocated
  fashion meaning
  \begin{align*}
    \rho_h(\xcoord) = \sum_{j \in \HypVertices} \rho_i
    \phi_(\xcoord) \, , \
    \mom_h(\xcoord) = \sum_{i \in \HypVertices} \mom_i
    \phi_i(\xcoord) \, , \
    \totme_h(\xcoord) = \sum_{i \in \HypVertices} \totme_i
    \phi_i(\xcoord) \, , \
  \end{align*}
  where $\rho_i \in \mathbb{R}$, $\mom_i \in \mathbb{R}^2$, $\totme_i \in
  \mathbb{R}$, and $\{\phi_i(\xcoord)\}_{i \in \HypVertices}$ is the basis
  of the scalar-valued finite element space $\FESpaceHypComp$ defined in
  \eqref{VdefSpace}.

\item[\itemizebullet] \textit{Conservation of linear invariants.} In the context
  of periodic boundary conditions the hyperbolic solver preserves the linear
  invariants:
  \begin{align}\label{consMasss}
    \sum_{i \in \HypVertices} m_i \rho_i^{n+1} = \sum_{i \in \HypVertices}
    m_i \rho_i^{n}
    \ , \ \
    \sum_{i \in \HypVertices} m_i \mom_i^{n+1} = \sum_{i \in \HypVertices}
    m_i \mom_i^{n}
    \ , \ \
    \sum_{i \in \HypVertices} m_i \totme_i^{n+1} = \sum_{i \in \HypVertices}
    m_i
    \totme_i^{n},
  \end{align}
  where $m_i$ was defined in \eqref{schemeMatrices}.

\item[\itemizebullet] \textit{Admissibility.} Assume that the initial data
  $\state_i^n = [\rho_i^n, \mom_i^n, \totme_i^n]^\transp$ is admissible, meaning
  $\state_i^n \in \mathcal{A}$ for all $i \in \HypVertices$, where the set
  $\mathcal{A}$ was defined in \eqref{DefAdmissibleA}. Then the updated
  state $\state_i^{n+1} = [\rho_i^{n+1}, \mom_i^{n+1}, \totme_i^{n+1}]^\transp$
  is admissible for all $i \in \HypVertices$ as well. We highlight that this is
  rather low requirements for preservation of pointwise properties. In general,
  positivity properties are not enough, and we might be interested in stronger
  properties, such as the preservation of the local minimum principle of the
  specific entropy, see \cite{Guer2018, Guer2019}.

\item[\itemizebullet] \textit{Entropy dissipation inequality. }We \emph{may}
  assume that the scheme preserves a global entropy inequality, meaning
  \begin{align}\label{EulerSolverDissipation}
    \sum_{i \in \HypVertices} m_i \eta(\state_i^{n+1}) \leq \sum_{i \in
    \HypVertices} m_i
    \eta(\state_i^{n}),
  \end{align}
  in the context of periodic boundary conditions.
\end{itemize}
For the sake of completeness, in \ref{sec:HypSolver} we make precise
the implementation details of the hyperbolic solver used in all our
computations. For any practical purpose, we may simply regard
$\texttt{euler\_system\_update}$ as the user's favorite choice of Euler scheme
(a black box) that is: consistent, conservative, it is mathematically
guaranteed not to crash, and may have some entropy-dissipation properties.

\begin{remark}[Partition of unity and consistent mass-matrix] We note that
  hyperbolic solvers using a consistent mass matrix do not satisfy
  conservation property \eqref{consMasss} directly, but rather the identities
  \begin{align}\label{ConsMassBalance}
    \sum_{j \in
    \HypVertices} m_{ij} \varrho_j^{n+1} = \sum_{j \in \HypVertices} m_{ij}
    \varrho_j^n,
  \end{align}
  where $\{\varrho_i^{n+1}\}_{i \in \HypVertices}$ represents a quantity of
  interest
  such as density, momentum or total mechanical energy, with $m_{ij}$ as
  defined in \eqref{schemeMatrices}. In this context, we consider the summation
  $\sum_{i \in \HypVertices}$ to both sides of identity
  \eqref{ConsMassBalance},
  using the partition of unity property $\sum_{i \in
    \HypVertices} m_{ij} = m_j$ (see Section \ref{sec:SpaceDisc}) we get
  \begin{align*}
    \sum_{i \in \HypVertices} \sum_{j \in \HypVertices}  m_{ij}
    \varrho_j^{n+1} =
    \sum_{i \in \HypVertices} \sum_{j \in \HypVertices}  m_{ij}
    \varrho_j^{n}
    \ \ \Rightarrow \ \
    \sum_{j \in \HypVertices}
    \underbrace{\big( \sum_{i \in \HypVertices} m_{ij} \big)}_{=
    \, m_j} \varrho_j^{n+1} =
    \sum_{j \in \HypVertices} \underbrace{\big(\sum_{i \in
    \HypVertices} m_{ij}\big)}_{=\,  m_j} \varrho_j^{n},
  \end{align*}
  recovering to the usual conservation identity.
\end{remark}

%%%%%%%%%%%%%%%%%%%%%%%%%%%%%%%%%%%%%%%%%%%%%%%%%%%%%%%%%%%%%%%%%%%%%%%%%%%%%%%%

%%%%%%%%%%%%%%%%%%%%%%%%%%%%%%%%%%%%%%%%%%%%%%%%%%%%%%%%%%%%%%%%%%%%%%%%%%%%%%%%

%%%%%%%%%%%%%%%%%%%%%%%%%%%%%%%%%%%%%%%%%%%%%%%%%%%%%%%%%%%%%%%%%%%%%%%%%%%%%%%%

%%%%%%%%%%%%%%%%%%%%%%%%%%%%%%%%%%%%%%%%%%%%%%%%%%%%%%%%%%%%%%%%%%%%%%%%%%%%%%%%

\subsection{Discretization of the Operator \#2}\label{sec:SourceUpdate}

This section concerns with the spatial discretization of Operator \#2, see
\eqref{Operator2}. We consider the following semi-discretation of
\eqref{SourceSmoothTest}: find $\vel_h \in [\FESpaceHypComp]^2$ and $\Hfield_h
\in \FESpaceH$ such that
\begin{align}\label{semiDiscSource}
\begin{split}
\langle \rho_h\partial_t\vel_h, \veltest_h \rangle &=
- \mu (\Hfield_h \times \curl{}\Hfield_h, \veltest_h ), \\
(\partial_t\Hfield_h, \Htest_h)
&=
(\Hfield_h \times \curl{}\Htest_h, \vel_h ) \, ,
\end{split}
\end{align}
for all $\veltest_h \in [\FESpaceHypComp]^2$ and $\Htest_h \in \FESpaceH$ more
precisely
\begin{align*}
  \vel_h = \sum_{i \in \HypVertices} \vel_i \phi_i
  \ \ \text{and} \ \
  \Hfield_h = \sum_{i \in \Hvertices} \HfieldDOF_i \HBasis_i,
\end{align*}
where $\{\phi_i\}_{i \in \HypVertices}$ is a basis of the
scalar-valued space $\FESpaceHypComp$ while $\{\HBasis_i\}_{i \in
  \Hvertices}$ is a vector-valued basis for the space $\FESpaceH$, see
Section \ref{sec:SpaceDisc} for the definition of the finite element spaces.
Note that the bilinear form containing the time-derivative of the velocity in
\eqref{semiDiscSource} is lumped, see \eqref{LumpedInner} for the definition of
lumping. This lumping is second-order accurate for the case of
first-order simplices (used in our computations), see Remark \ref{QuadsLumping}
for the case of quadrilateral elements.

\begin{remark}[Lumping and higher order elements]\label{QuadsLumping}
In all our computations we use simplices. However, if we were to use tensor
product elements (\eg quadrilaterals) mass lumping has consistency order $2 k -
3$ when using $\mathbb{Q}^k$ elements with Gauss-Lobatto interpolation
points. Therefore, mass-lumping preserves the formal consistency order of the
method and is compatible with arbitrarily high-order polynomial degree.
\end{remark}

\begin{lemma}[Conserved quantities]\label{ConseQuantRemark} The semi-discrete
  scheme \eqref{semiDiscSource}, as well as the fully discrete scheme using
  Crank-Nicolson method
  \begin{subequations}\label{SourceUpdateCN}
    \begin{align}
     \langle \rho_h^n (\vel_h^{n+1} - \vel_h^{n}), \veltest_h \rangle &=
    - \dt_n \mu \int_{\Omega} (\Hfield_h^{n+\frac{1}{2}} \times
     \curl{}\Hfield_h^{n+\frac{1}{2}})\cdot \veltest_h \dx \, , \\
     \label{SourceUpdateCNInduc}
     \int_{\Omega} (\Hfield_h^{n+1} - \Hfield_h^{n})\cdot \Htest_h
     &=
     \dt_n \int_{\Omega}  (\Hfield_h^{n+\frac{1}{2}} \times
     \curl{}\Htest_h)\cdot
     \vel_h^{n+\frac{1}{2}} \dx \, ,
    \end{align}
  \end{subequations}
  where $\vel_h^{n+\frac{1}{2}} := \frac{1}{2}(\vel_h^{n} + \vel_h^{n+1})$
  and $\Hfield_h^{n+\frac{1}{2}} := \frac{1}{2}(\Hfield_h^{n} +
\Hfield_h^{n+1})$,
  preserve the energy identity
  \begin{align}\label{EnergyDiscSource}
    \sum_{i \in \HypVertices}
    \tfrac{m_i}{2}\dens_i^{n+1} |\vel_i^{n+1}|^2
    + \tfrac{\mu}{2}  \|\Hfield_h^{n+1}\|_{L^2(\Omega)}^{2}
    =
    \sum_{i \in \HypVertices}
    \tfrac{m_i}{2}\dens_i^{n} |\vel_i^{n}|^2
    + \tfrac{\mu}{2}  \|\Hfield_h^{n}\|_{L^2(\Omega)}^{2}.
  \end{align}
  % satisfy
  % satisfies \eqref{EnergyDiscSource} as well.
\end{lemma}

\begin{proof} We consider the proof of the fully discrete scheme
  \eqref{SourceUpdateCN}. We take $\veltest_h = \vel_h^{n+\frac{1}{2}}$
  and $\Htest_h = \Hfield_h^{n+\frac{1}{2}}$ in \eqref{SourceUpdateCN}, the
  result follows by noting that
  \begin{align*}
    \langle \rho_h^n (\vel_h^{n+1} - \vel_h^{n}), \vel_h^{n+\frac{1}{2}} \rangle
=
    \sum_{i \in \HypVertices}  \tfrac{m_i}{2}\rho_i^n (\vel_i^{n+1} -
    \vel_i^{n})\cdot(\vel_i^{n+1} + \vel_i^{n}),
  \end{align*}
  using difference of squares, and adding both lines leading to the cancellation
  of right hand side terms.
\end{proof}

\begin{proposition}[Preservation of the weak-divergence]\label{WeakDivProp}
  Assume that $\FESpaceH$ is the curl-conforming $\text{BDM}_1$ space,
  as defined in \eqref{BDMspace}. Then, we have that the solution of
  \eqref{semiDiscSource} and \eqref{SourceUpdateCN} will satisfy
  \begin{align}\label{DiscFree}
    (\Hfield_h^{n+1}, \nabla\omega_h) = (\Hfield_h^n, \nabla\omega_h)
  \end{align}
  for all $\omega_h \in \FESpacePot$, with $ \FESpacePot$ as defined in
  \eqref{PotSpace}. Note that \eqref{DiscFree} is nothing else than the discrete
  counterpart of weak divergence property.
\end{proposition}

\begin{proof} The proof follows from the fact that the inclusion
  $\nabla\FESpacePot \subset \FESpaceH$ holds true, therefore $\nabla\omega_h$
  is a valid test function of \eqref{SourceUpdateCNInduc} (for all $\omega_h
  \in \FESpacePot$). Inserting this test function into
  \eqref{SourceUpdateCNInduc} we get
  \begin{align*}
    \int_{\Omega} (\Hfield_h^{n+1} - \Hfield_h^{n})\cdot \nabla\omega_h
    &=
      \dt_n \int_{\Omega}  (\Hfield_h^{n+\frac{1}{2}} \times
      \curl{}\nabla\omega_h)\cdot
      \vel_h^{n+\frac{1}{2}} \dx \, ,
  \end{align*}
  where the right hand side is zero since $\curl{}\nabla\omega_h \equiv
  \bzero$.
\end{proof}

Scheme \eqref{SourceUpdateCN} defines the numerical procedure
used to update the momentum and magnetic field during the evolution of Stage
\#2: we summarize its implementation in Algorithm \ref{VelAndHcnUpdate} as a
function with input and return arguments. However, Algorithm
\ref{VelAndHcnUpdate} does not prescribe the evolution of the density $\rho$ and
total mechanical energy $\totme$. We observe in \eqref{Operator2rev} that the
density does not evolve during the evolution of
Stage \#2. On other hand, we will use \eqref{TotmeEvoRule} in order to
update the total mechanical energy. We summarize the entire update
for Stage \#2 in Algorithm \ref{SourceCNalg}.

\begin{algorithm}[H]
  \caption{\texttt{momentum\_and\_h\_field\_update}($
    \{\rho_h^n, \mom_h^{n}, \Hfield_h^{n}, \dt_n \}$)}
  \label{VelAndHcnUpdate}
  \begin{align*}
    &\texttt{Define:}\, \vel_h^{n} := \sum_{i \in \HypVertices}
      \tfrac{\mom_i^{n}}{\rho_i^n} \phi_i \\
    &\texttt{Find:} \{\vel_h^{n+1},\Hfield_h^{n+1}\} \in [\FESpaceHypComp]^2
\times
      \FESpaceH \ \texttt{such that} \\
    &\ \ \ \ \ \ \ \ \langle \rho_h^n (\vel_h^{n+1} - \vel_h^{n}), \veltest_h
      \rangle =
      - \dt_n \mu \int_{\Omega} (\Hfield_h^{n+\frac{1}{2}} \times
      \curl{}\Hfield_h^{n+\frac{1}{2}})\cdot \veltest_h \dx \,
      , \\
    &\ \ \ \ \ \int_{\Omega} (\Hfield_h^{n+1} - \Hfield_h^{n})\cdot \Htest_h
      =
      \dt_n \int_{\Omega}  (\Hfield_h^{n+\frac{1}{2}} \times
      \curl{}\Htest_h)\cdot
      \vel_h^{n+\frac{1}{2}} \dx \\
    &\texttt{where }\vel_h^{n+\frac{1}{2}} := \tfrac{1}{2}(\vel_h^{n} +
      \vel_h^{n+1}) \texttt{ and } \Hfield_h^{n+\frac{1}{2}} :=
      \tfrac{1}{2}(\Hfield_h^{n} + \Hfield_h^{n+1}). \\
    &\texttt{Define:}\, \mom_h^{n+1} := \sum_{i \in \HypVertices}
      (\vel_i^{n+1} \rho_i^n) \phi_i \\
    &\texttt{Return:} \{\mom_h^{n+1}, \Hfield_h^{n+1}\}
  \end{align*}
\end{algorithm}

\begin{algorithm}[H]
  \caption{\texttt{source\_update}($\{\rho_h^n, \mom_h^{n},
    \totme_h^n, \Hfield_h^{n}, \dt_n\}$)}
  \label{SourceCNalg}
  \begin{align}
    \nonumber
    &\{\mom_h^{n+1}, \Hfield_h^{n+1}\} :=
      \texttt{momentum\_and\_h\_field\_update}(
      \{\rho_h^n, \mom_h^{n}, \Hfield_h^{n}, \dt_n \}) \\
    \nonumber
    &\texttt{for } i \in \HypVertices \\
    \label{UpdateDens}
    &\ \ \dens_i^{n+1} := \dens_i^{n} \\
    \label{UpdateTotMe}
    &\ \ \totme_i^{n+1} := \totme_i^{n} - \tfrac{1}{2 \dens_i^n}
      |\mom_i^{n}|^2 + \tfrac{1}{2 \dens_i^{n+1}} |\mom_i^{n+1}|^2 \\
    \nonumber
    &\texttt{end for} \\
    \nonumber
    &\texttt{Return:} \{\rho_h^{n+1}, \mom_h^{n+1}, \totme_h^{n+1},
      \Hfield_h^{n+1}\}
  \end{align}
  \textit{Comments: the method }
  \texttt{momentum\_and\_h\_field\_update} \textit{ is
    described in Algorithm \ref{VelAndHcnUpdate}.}
\end{algorithm}

\begin{lemma}[Properties preserved by Algorithm
  \ref{SourceCNalg}]\label{SourceLemma} The
  scheme $\texttt{source\_update}$, described by Algorithm
  \ref{SourceCNalg}, preserves the following global energy
  \begin{align}
    \label{CNconservation}
    \sum_{i \in \HypVertices} m_i \totme_i^{n+1}
    + \tfrac{\mu}{2} \|\Hfield_h^{n+1}\|_{L^2(\Omega)}^2
    =
    \sum_{i \in \HypVertices} m_i \totme_i^{n}
    + \tfrac{\mu}{2} \|\Hfield_h^{n}\|_{L^2(\Omega)}^{2},
  \end{align}
  as well as pointwise properties
  \begin{align*}
    \inte(\state_i^{n+1}) = \inte(\state_i^{n}) , \
    s(\state_i^{n+1}) = s(\state_i^{n}) , \
    \entropy(\state_i^{n+1}) = \entropy(\state_i^{n})
    \text{ for all } i \in \HypVertices \, ,
  \end{align*}
  with $\state_i^{n} = [\rho_i^n, \mom_i^n, \totme_i^n]^\transp$. In particular,
  this implies that the following global property for the mathematical entropy
  \begin{align}\label{GlobalEntropy}
    \sum_{i \in \HypVertices} m_ i \entropy(\state_i^{n+1}) = \sum_{i \in
    \HypVertices} m_ i \entropy(\state_i^{n}).
  \end{align}
\end{lemma}

\begin{proof} From Lemma \ref{ConseQuantRemark} we know that
  \begin{align}\label{SourceUpStab}
    \sum_{i \in \HypVertices}
    \tfrac{m_i}{2 \dens_i^{n+1}} |\mom_i^{n+1}|^2
    + \tfrac{\mu}{2 } \|\Hfield_h^{n+1}\|_{L^2(\Omega)}^{2}
    =
    \sum_{i \in \HypVertices}
    \tfrac{m_i}{2 \dens_i^{n}} |\mom_i^{n}|^2
    + \tfrac{\mu}{2 } \|\Hfield_h^{n}\|_{L^2(\Omega)}^{2}.
  \end{align}
  Multiplying \eqref{UpdateTotMe} by $m_i$, reorganizing, and
  adding for all nodes we get
  \begin{align*}
    \sum_{i \in \HypVertices} m_i \big( \totme_i^{n+1} - \tfrac{1}{2
    \dens_i^{n+1}} |\mom_i^{n+1}|^2 \big)
    =
    \sum_{i \in \HypVertices} m_i \big( \totme_i^n - \tfrac{1}{2 \dens_i^{n}}
    |\mom_i^{n}|^2 \big).
  \end{align*}
  Adding this last result to both sides of \eqref{SourceUpStab} yields
  \eqref{CNconservation}. Note that, \eqref{UpdateTotMe} implies pointwise
  invariance of the internal energy $\inte(\state)$ by construction,
  which combined with the invariance of the density \eqref{UpdateDens} are
  enough to guarantee pointwise preservation of the specific and mathematical
  entropy. Finally, \eqref{GlobalEntropy} follows from the pointwise
  preservation of the mathematical entropy.
\end{proof}

%%%%%%%%%%%%%%%%%%%%%%%%%%%%%%%%%%%%%%%%%%%%%%%%%%%%%%%%%%%%%%%%%%%%%%%%%%%%%%%%
%%%%%%%%%%%%%%%%%%%%%%%%%%%%%%%%%%%%%%%%%%%%%%%%%%%%%%%%%%%%%%%%%%%%%%%%%%%%%%%%
%%%%%%%%%%%%%%%%%%%%%%%%%%%%%%%%%%%%%%%%%%%%%%%%%%%%%%%%%%%%%%%%%%%%%%%%%%%%%%%%
%%%%%%%%%%%%%%%%%%%%%%%%%%%%%%%%%%%%%%%%%%%%%%%%%%%%%%%%%%%%%%%%%%%%%%%%%%%%%%%%
%%%%%%%%%%%%%%%%%%%%%%%%%%%%%%%%%%%%%%%%%%%%%%%%%%%%%%%%%%%%%%%%%%%%%%%%%%%%%%%%
%%%%%%%%%%%%%%%%%%%%%%%%%%%%%%%%%%%%%%%%%%%%%%%%%%%%%%%%%%%%%%%%%%%%%%%%%%%%%%%%
%%%%%%%%%%%%%%%%%%%%%%%%%%%%%%%%%%%%%%%%%%%%%%%%%%%%%%%%%%%%%%%%%%%%%%%%%%%%%%%%

\subsection{The MHD update scheme}\label{sec:FirstScheme}

The Marchuck-Strang splitting scheme involves three steps: the first one using
full time-step $\dt_n$, advancing Operator \#1 described in \eqref{Operator1},
the second step using a double size time-step $2 \dt_n$ evolving in time the
Operator \#2 described in \eqref{Operator2rev}, and a third step using a full
time-step $\dt_n$ evolving Operator \#1 again. We summarize the scheme in
Algorithm \eqref{MHDupdateAlg}.
\begin{algorithm}[H]
  \caption{\texttt{mhd\_update}($\{\rho_h^n, \mom_h^{n},
    \totme_h^n, \Hfield_h^{n}\}$)}
  \label{MHDupdateAlg}
  \begin{align}
    \nonumber
    \{\rho_h^1, \mom_h^{1}, \totme_h^1, \dt_n\} &:=
    \texttt{euler\_system\_update}(\{\rho_h^n, \mom_h^{n}, \totme_h^n\}) \\
    \nonumber
    \Hfield_h^{1} &:= \Hfield_h^n \\
    \nonumber
    \{\rho_h^2, \mom_h^{2}, \totme_h^2, \Hfield_h^2\} &:=
    \texttt{source\_update}
    (\{\rho_h^1, \mom_h^{1}, \totme_h^1, \Hfield_h^{1}, 2 \dt_n\}) \\
    \nonumber
    \{\rho_h^{n+1}, \mom_h^{n+1}, \totme_h^{n+1}\}
    &:=
    \texttt{euler\_system\_update}
    (\{\rho_h^2, \mom_h^{2}, \totme_h^2, \dt_n\}) \\
    \nonumber
    \Hfield_h^{n+1} &:= \Hfield_h^2 \\
    \nonumber
    \texttt{Return}&: \{\rho_h^{n+1}, \mom_h^{n+1}, \totme_h^{n+1},
                      \Hfield_h^{n+1}\}
  \end{align}
  \textit{Comments: the procedure }\texttt{euler\_system\_update} \textit{
    represents an Euler's solver satisfying the assumptions described in Section
    \ref{sec:hypminimal}. On the other hand, the implementation of}
    \texttt{source\_update} \textit{is described in Algorithm \ref{SourceCNalg}.
    Note that the time-step size $\dt_n$ is determined during the
    first call to }\texttt{\texttt{euler\_system\_update}}\textit{, then the
    method} \texttt{source\_update} \textit{and the second call to}
    \texttt{\texttt{euler\_system\_update}} \textit{have to comply with such
    time-step size. Finally, note that the return argument
    $\state_h^{n+1} = [\rho_h^{n+1}, \mom_h^{n+1}, \totme_h^{n+1},
    \Hfield_h^{n+1}]^\transp$ represents the solution at time $t^{n+1} =
    t^n + 2\dt^n$.}
\end{algorithm}

\begin{proposition}[Properties preserved by \texttt{mhd\_update}]
  Assuming periodic boundary conditions, and that the Euler scheme underlying
  $\texttt{euler\_system\_update}$ satisfies the assumptions described in
Section
  \ref{sec:hypminimal}, then the procedure $\texttt{mhd\_update}$ described by
  Algorithm \ref{MHDupdateAlg} preserves the following global estimate
  \begin{align*}
    \sum_{i \in \HypVertices} m_i \totme_i^{n+1}
    + \tfrac{\mu}{2} \|\Hfield_h^{n+1}\|_{L^2(\Omega)}^2
    &=
      \sum_{i \in \HypVertices} m_i \totme_i^{n}
      + \tfrac{\mu}{2} \|\Hfield_h^{n}\|_{L^2(\Omega)}^{2} \ ,
  \end{align*}
  as well as pointwise admissibility $\state_i^{n+1} \in \mathcal{A}$ for all
  $i \in \HypVertices$ with $\mathcal{A}$ as defined in \eqref{DefAdmissibleA}.
  The scheme also preserves the weak divergence property $(\Hfield_h^{n+1},
  \nabla\omega_h) = (\Hfield_h^n, \nabla\omega_h)$ for all $\omega_h \in
  \FESpacePot$. If in addition, we assume that property
  \eqref{EulerSolverDissipation} for $\texttt{euler\_system\_update}$ holds, then
  we also have the global entropy estimate:
  \begin{align*}
    \sum_{i \in \HypVertices} m_i \eta(\state_i^{n+1}) \leq \sum_{i \in
    \HypVertices} m_i
    \eta(\state_i^{n}).
  \end{align*}
\end{proposition}

\begin{proof} Key results were proven in Corollary Lemma
  \ref{ConseQuantRemark}, Proposition \ref{WeakDivProp}, and Lemma
  \ref{SourceLemma}. Under the assumptions of the lemma, each stage of Strang
  splitting preserves energy-stability, admissibility, weak-divergence,
  and entropy stability. The proof follows from the sequential nature of operator
  splitting and the assumptions on $\texttt{euler\_system\_update}$ described
  Section \ref{sec:hypminimal}.
\end{proof}

%%%%%%%%%%%%%%%%%%%%%%%%%%%%%%%%%%%%%%%%%%%%%%%%%%%%%%%%%%%%%%%%%%%%%%%%%%%%%%%%
%%%%%%%%%%%%%%%%%%%%%%%%%%%%%%%%%%%%%%%%%%%%%%%%%%%%%%%%%%%%%%%%%%%%%%%%%%%%%%%%
%%%%%%%%%%%%%%%%%%%%%%%%%%%%%%%%%%%%%%%%%%%%%%%%%%%%%%%%%%%%%%%%%%%%%%%%%%%%%%%%
%%%%%%%%%%%%%%%%%%%%%%%%%%%%%%%%%%%%%%%%%%%%%%%%%%%%%%%%%%%%%%%%%%%%%%%%%%%%%%%%
%%%%%%%%%%%%%%%%%%%%%%%%%%%%%%%%%%%%%%%%%%%%%%%%%%%%%%%%%%%%%%%%%%%%%%%%%%%%%%%%

\section{Numerical experiments}\label{sec:numexp}
In this section, we demonstrate the capability of the proposed scheme through
several numerical experiments. Second-order accuracy for smooth problems is
confirmed in Section~\ref{sec:numerical:accuracy}. The results for the popular
1D Riemann Brio-Wu problem \cite{Brio_Wu_1988} is reported in
Section~\ref{sec:numerical:briowu}. In Sections~\ref{sec:numerical:blast} and
\ref{sec:numerical:jet}, we look at the two challenging MHD benchmarks: the
blast problem \cite{Balsara_1999}, and the astrophysical jet problem
\cite{Wu2018}.
\subsection{Accuracy test: smooth isentropic vortex \cite{Wu2018}}
\label{sec:numerical:accuracy}
The computational domain is a square $[-10, 10] \times [-10, 10]$. We start
with the ambient solution: the velocity $\vel_{\infty}=(1,1)^\top$, the magnetic
field $\Bfield_{\infty}=(0,0)^\top$, and the pressure $p_{\infty}=1$. At each
spatial point $(x_0, x_1)^\top$, we define the radius from it to the origin
$r=\sqrt{x_0^2+x_1^2}$. The vortex is initialized by adding smooth pertubations
to the velocity, magnetic field, and pressure of the ambient solution,
\begin{align*}
  \vel = \vel_{\infty}+(-x_1, x_0)\delta v, \quad \Bfield =
  \Bfield_{\infty}+(-x_1, x_0)\delta B, \quad p = p_{\infty}+(-x_1, x_0)\delta
p,
\end{align*}
where
\begin{align*}
  \delta v=\tfrac{\kappa}{2\pi}\mathrm{e}^{0.5(1-r^2)}, \quad
  \delta B = \tfrac{\mu}{2\pi}\mathrm{e}^{0.5(1-r^2)}, \quad
  \delta p = \tfrac{\mu^2(1-r^2)-\kappa^2}{8\pi^2}\mathrm{e}^{1-r^2}.
\end{align*}
The real numbers $\mu$ and $\kappa$ are vortex strength parameters. In this
test, we set $\kappa=\sqrt{2}\mu$ similar to \cite{Wu2018}.
The final time is $T=0.05$. The CFL number is chosen to be 0.1. The adiabatic
gas constant is $\gamma = \frac{5}{3}$. The convergence results with $\mu=1$ is
presented in Table~\ref{table:convergence_smooth1}, $\mu=5.38948943$ in
Table~\ref{table:convergence_smooth9}, and $\mu=5.389489439$ in
Table~\ref{table:convergence_smooth_P12}. In the second and the third case,
the pressure at the vortex centre is very close to zero: $3.3\times10^{-9}$ in
the second case, and $5.3\times10^{-12}$ in the third case. We want to examine
how this affects the convergence rates.

\begin{table}[h!]
  \centering
  \caption{Convergence of velocity and magnetic field on smooth solutions.
  Parameter $\mu=1.0$.}
  \label{table:convergence_smooth1}
  \begin{tabular}{|c|c|c|c|c|c|c|c|}
    \hline
    {} &    \#DOFs  &   L$^1$    &   Rate &    L$^2$   &   Rate & L$^\infty$   &
                                                                                 Rate \\ \hline\hline
    \multirow{4}{*}{\rotatebox[origin=c]{90}{$\vel_h$}}
       &     1922 &   4.27E-04 &     -- &   2.34E-03 &     -- &   2.70E-02 &     -- \\
       &     7442 &   1.07E-04 &   2.04 &   5.98E-04 &   2.01 &   7.33E-03 &   1.93 \\
       &    29282 &   2.63E-05 &   2.05 &   1.47E-04 &   2.04 &   1.84E-03 &   2.02 \\
       &   116162 &   6.30E-06 &   2.08 &   3.55E-05 &   2.06 &   4.47E-04 &   2.05 \\ \hline
    \multirow{4}{*}{\rotatebox[origin=c]{90}{$\Bfield_h$}}
       &     5520 &   6.47E-02 &     -- &   7.41E-02 &     -- &   2.77E-02 &     -- \\
       &    21840 &   1.62E-02 &   2.01 &   1.88E-02 &   1.99 &   7.43E-03 &   1.91 \\
       &    86880 &   4.06E-03 &   2.01 &   4.72E-03 &   2.00 &   1.85E-03 &   2.02 \\
       &   346560 &   1.02E-03 &   2.00 &   1.18E-03 &   2.00 &   4.59E-04 &   2.01 \\ \hline
  \end{tabular}
\end{table}

\begin{table}[h!]
  \centering
  \caption{Convergence of velocity and magnetic field on smooth solutions.
  Parameter $\mu=5.38948943$. The minimum pressure is $3.3\times10^{-9}$.}
  \label{table:convergence_smooth9}
  \begin{tabular}{|c|c|c|c|c|c|c|c|}
    \hline
    {} &    \#DOFs  &   L$^1$    &   Rate &    L$^2$   &   Rate & L$^\infty$   &
                                                                                 Rate \\ \hline\hline
    \multirow{4}{*}{\rotatebox[origin=c]{90}{$\vel_h$}}
       &     1922 &   2.99E-03 &     -- &   1.63E-02 &     -- &   1.18E-01 &     -- \\
       &     7442 &   7.69E-04 &   2.00 &   4.54E-03 &   1.89 &   3.92E-02 &   1.63 \\
       &    29282 &   1.73E-04 &   2.18 &   1.03E-03 &   2.17 &   1.27E-02 &   1.64 \\
       &   116162 &   3.92E-05 &   2.15 &   2.22E-04 &   2.23 &   3.15E-03 &   2.02 \\ \hline
    \multirow{4}{*}{\rotatebox[origin=c]{90}{$\Bfield_h$}}
       &     5520 &   6.48E-02 &     -- &   7.41E-02 &     -- &   3.15E-02 &     -- \\
       &    21840 &   1.63E-02 &   2.01 &   1.89E-02 &   1.99 &   8.88E-03 &   1.84 \\
       &    86880 &   4.07E-03 &   2.01 &   4.73E-03 &   2.00 &   2.07E-03 &   2.11 \\
       &   346560 &   1.02E-03 &   2.00 &   1.18E-03 &   2.00 &   5.10E-04 &   2.02 \\ \hline
  \end{tabular}
\end{table}

\begin{table}[h!]
  \centering
  \caption{Convergence of velocity and magnetic field on smooth solutions.
  Parameter $\mu=5.389489439$. The minimum pressure is $5.3\times10^{-12}$.}
  \label{table:convergence_smooth_P12}
  \begin{tabular}{|c|c|c|c|c|c|c|c|}
    \hline
    {} &    \#DOFs  &   L$^1$    &   Rate &    L$^2$   &   Rate & L$^\infty$   &
                                                                                 Rate \\ \hline\hline
    \multirow{4}{*}{\rotatebox[origin=c]{90}{$\vel_h$}}
       &     1922 &   2.89E-03 &     -- &   1.63E-02 &     -- &   1.17E-01 &     -- \\
       &     7442 &   7.30E-04 &   2.03 &   4.47E-03 &   1.91 &   3.92E-02 &   1.62 \\
       &    29282 &   1.66E-04 &   2.16 &   1.06E-03 &   2.10 &   1.48E-02 &   1.42 \\
       &   116162 &   3.86E-05 &   2.12 &   2.55E-04 &   2.07 &   5.43E-03 &   1.46 \\ \hline
    \multirow{4}{*}{\rotatebox[origin=c]{90}{$\Bfield_h$}}
       &     5520 &   6.48E-02 &     -- &   7.41E-02 &     -- &   3.27E-02 &     -- \\
       &    21840 &   1.63E-02 &   2.01 &   1.89E-02 &   1.99 &   9.15E-03 &   1.85 \\
       &    86880 &   4.08E-03 &   2.01 &   4.74E-03 &   2.00 &   3.01E-03 &   1.61 \\
       &   346560 &   1.02E-03 &   2.00 &   1.19E-03 &   2.00 &   1.22E-03 &   1.31 \\ \hline
  \end{tabular}
\end{table}
Overall, we obtain second-order accuracy in both $L^1$- and $L^2$- norms.
However, we note that $L^\infty$ rates in
Table~\ref{table:convergence_smooth_P12} are not sharp. This is
expected, since the results of Table~\ref{table:convergence_smooth_P12},
with a vacuum of $\mathcal{O}(10^{-12})$, are on the limit of what
can be meaningfully computed using double precision accuracy. For instance,
with such a strong vacuum, the accuracy of the map $\mom \mapsto \vel :=
\frac{\mom}{\rho}$, or even the computation of the internal energy, are just a
big stretch from reasonable expectations. We also notice that numerical linear
algebra technology starts to break down at such limits as well. For instance,
computational practice shows that, in the context of large non-symmetric
systems, it is nearly impossible to enforce relative tolerance of Krylov
methods much smaller than $\mathcal{O}(10^{-13})$. These errors propagate from
the solution of the source-system \eqref{SourceUpdateCN} into the rest of the
scheme. We have verified that by setting the slightly weaker vacuum of
$\mathcal{O}(10^{-11})$, we immediately recover sharp second-order rates in
$L^\infty$-norm.

\subsection{1D Riemann problem: Brio-Wu \cite{Brio_Wu_1988}}
\label{sec:numerical:briowu}
The Brio-Wu problem is a popular 1D benchmark for MHD schemes. The domain is
$[0,1]$. The initial solution is
\begin{align*}
  [\rho, \vel, p, \Bfield]^\transp =
  \begin{cases}
    [1, 0,0,1,0.75,1]^\transp, & x \in [0, 0.5), \\
    [0.125, 0,0,0.1,0.75,-1]^\transp, & x \in [0.5, 1].
  \end{cases}
\end{align*}
The adiabatic gas constant $\gamma=2$. The final time $T=0.1$. CFL number 0.1
is used. Despite the nonuniqueness of the Riemann solution, the solutions
obtained by almost all numerical schemes converge to a specific irregular
solution. We refer the interested readers to \cite{Takahashi_2013} and
references therein. To calculate the convergence rates, we compute a reference
solution using the Athena code \cite{Stone_2008} with 10000 grid intervals. The
density solution and the $y$-component of the magnetic solution when first-order
viscosity is used are shown in Figure~\ref{fig:briowu_firstorder}. The obtained
solution is smooth in all the components, including the magnetic field although
no regularization is added to it. High-order entropy-based viscosity is then
employed to lower the error. The convergence behavior of our numerical scheme is
shown in Table~\ref{table:convergence_briowu_density}. The density solution and
the $y$-component of the magnetic solution are shown in Figure~\ref{fig:briowu}.
From the convergence tables~\ref{table:convergence_briowu_density} and the
solution plots in Figure~\ref{fig:briowu}, our solution converges towards the
reference solution.

\begin{figure}[h!]
\centering
\begin{subfigure}{0.49\textwidth}
\centering
\includegraphics[width=\textwidth]{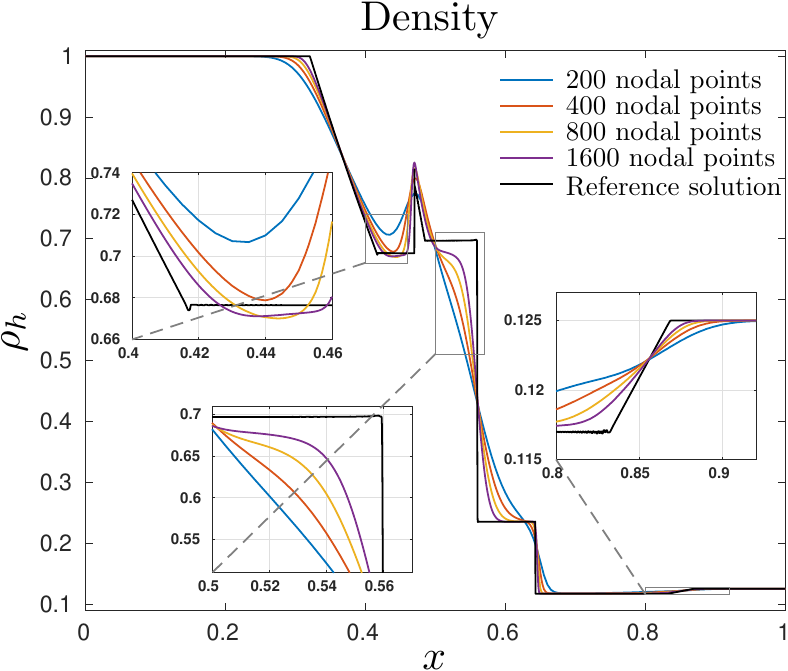}
\end{subfigure}
\hfill
\begin{subfigure}{0.49\textwidth}
\centering
\includegraphics[width=\textwidth]{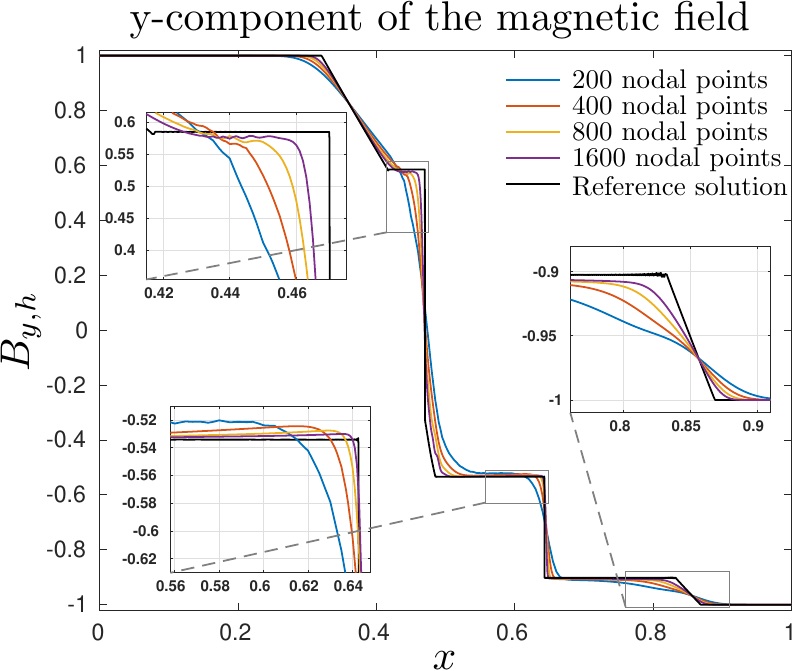}
\end{subfigure}
\caption{Density solution $\rho_h$ and the y-component of the magnetic field
$\Bfield_{y,h}$ solution to the Brio-Wu problem using first-order viscosity.
Note that even though the magnetic field has no viscous stabilization, the
solution shows no significant unphysical oscillation.}
\label{fig:briowu_firstorder}
\end{figure}

\begin{figure}[h!]
  \centering
  \begin{subfigure}{0.49\textwidth}
    \centering
    \includegraphics[width=\textwidth]{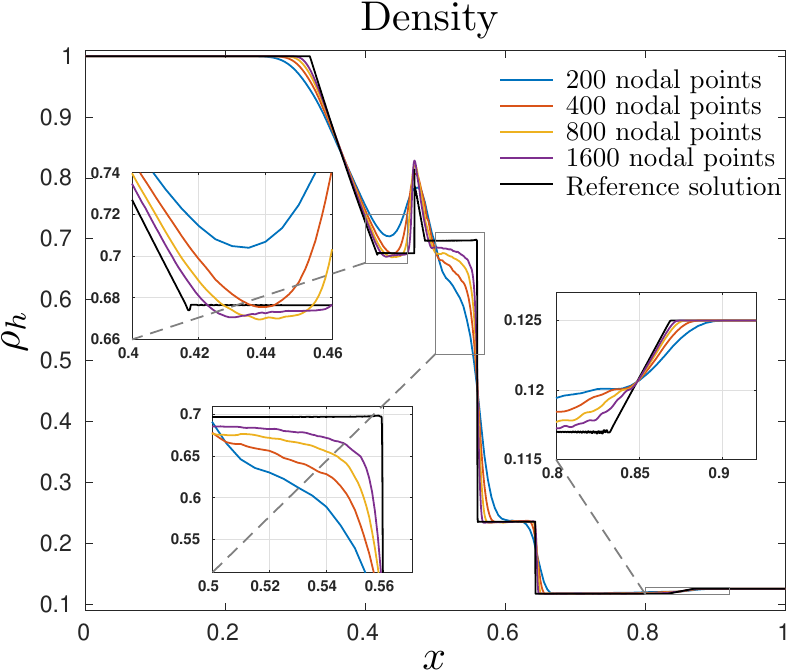}
  \end{subfigure}
  \hfill
  \begin{subfigure}{0.49\textwidth}
    \centering
    \includegraphics[width=\textwidth]{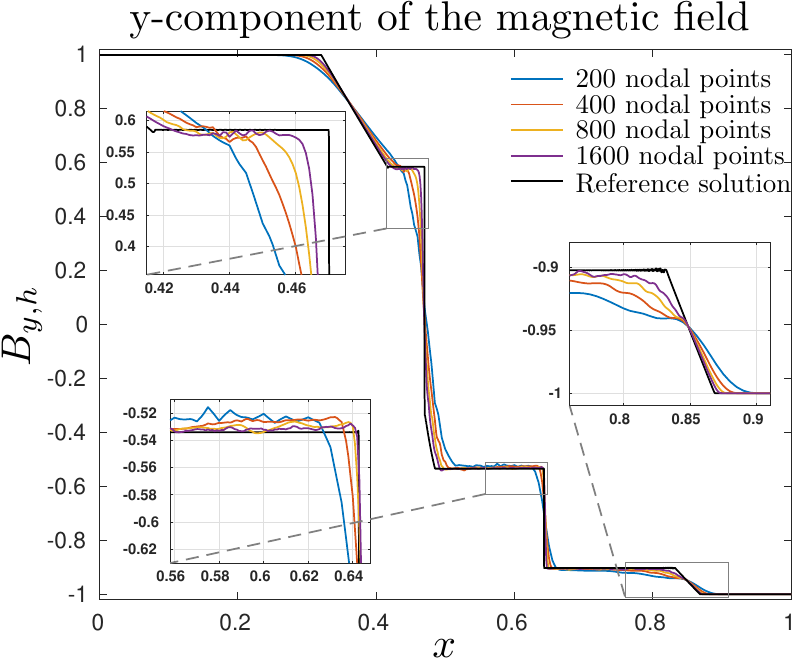}
  \end{subfigure}
\caption{Density solution $\rho_h$ and the y-component of the magnetic
field $\Bfield_{y,h}$ solution to the Brio-Wu problem using high-order entropy
viscosity method. Just like any other second-order (or higher) accuracy scheme,
some oscillations are to be expected. However, the natural expectation is that
unphysical oscillations do not persist further refinement.}
\label{fig:briowu}
\end{figure}

\begin{table}[h!]
  \centering
  \caption{Convergence rates of the density solution of the Brio-Wu problem
  using high-order entropy viscosity method}
  \label{table:convergence_briowu_density}
  \begin{tabular}{|c|c|c|c|c|}
    \hline
    \#nodes  &   L$^1$    &   Rate &    L$^2$&   Rate\\ \hline\hline
    100 &   3.11E-02 &     -- &   5.36E-02 &     -- \\
    200 &   1.89E-02 &   0.73 &   3.91E-02 &   0.46 \\
    400 &   1.17E-02 &   0.69 &   2.95E-02 &   0.41 \\
    800 &   7.17E-03 &   0.71 &   2.19E-02 &   0.43 \\
    1600 &   4.43E-03 &   0.69 &   1.62E-02 &   0.44\\ \hline
  \end{tabular}
\end{table}

\subsection{MHD Blast problem \cite{Balsara_1999}}
\label{sec:numerical:blast}

The domain is a square $[-0.5, 0.5]\times[-0.5, 0.5]$. Periodic boundary
conditions are imposed in both $x-$ and $y-$ directions. The solution is
initialized with an ambient fluid,
\begin{align*}
  [\rho, \vel, p, \Bfield]^\transp =
  \big[1, 0,0, 0.1, \tfrac{100}{\sqrt{4\pi}}, 0\big]^\transp.
\end{align*}
In the circle centered at the origin with radius $R=0.1$, the pressure is
initialized at $p=1000$. The 10 000 times pressure difference creates a strong
blast effect that is difficult for the numerical methods to capture. That is
because the pressure can easily become negative due to approximation errors.
The adiabatic gas constant is $\gamma=1.4$. The solution at $T=0.01$ is plotted
in Figure~\ref{fig:Blast}. The obtained solutions agree well with the existing
references, \eg \cite{Balsara_1999,Wu2018,Wu2019}. Detailed structures of the
solution are visible, and no oscillatory behaviors are observed.

\begin{figure}[h!]
  \centering
  \begin{subfigure}{0.48\textwidth}
    \centering
    \includegraphics[width=\textwidth]{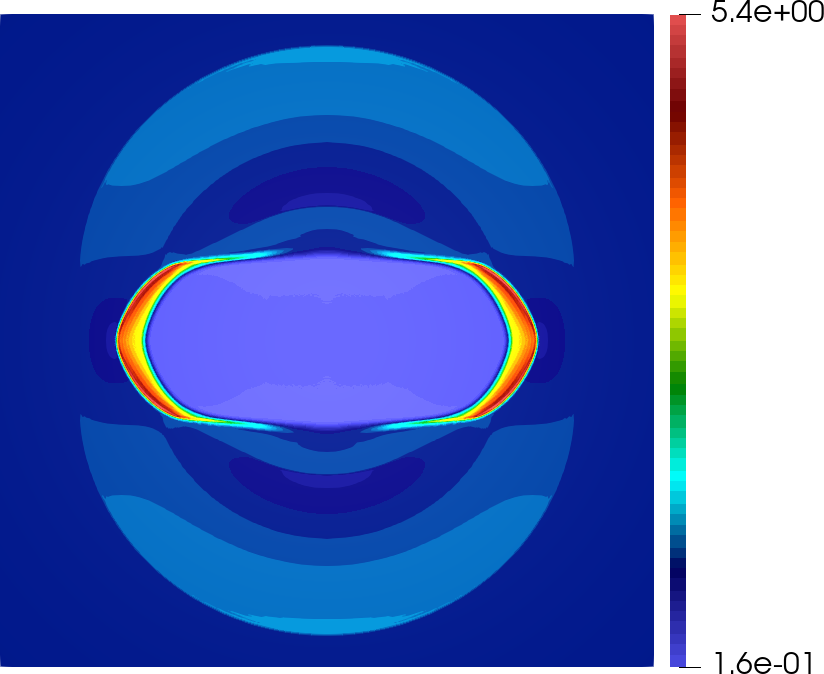}
    \caption{Density $\rho_h$}
  \end{subfigure}
  \hfill
  \begin{subfigure}{0.48\textwidth}
    \centering
    \includegraphics[width=\textwidth]{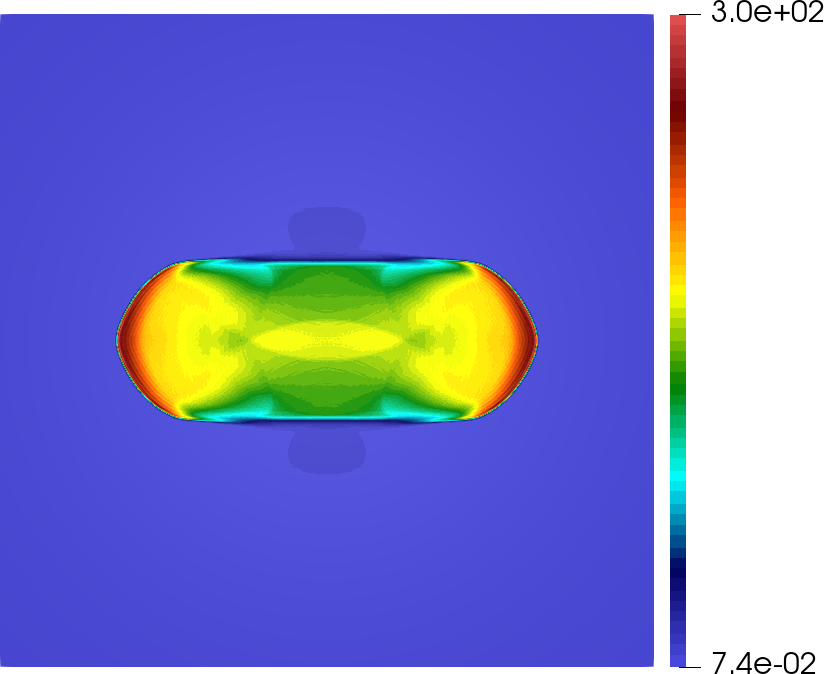}
    \caption{Hydrodynamic pressure $p_h$}
  \end{subfigure}
  \hfill
  \begin{subfigure}{0.48\textwidth}
    \centering
    \includegraphics[width=\textwidth]{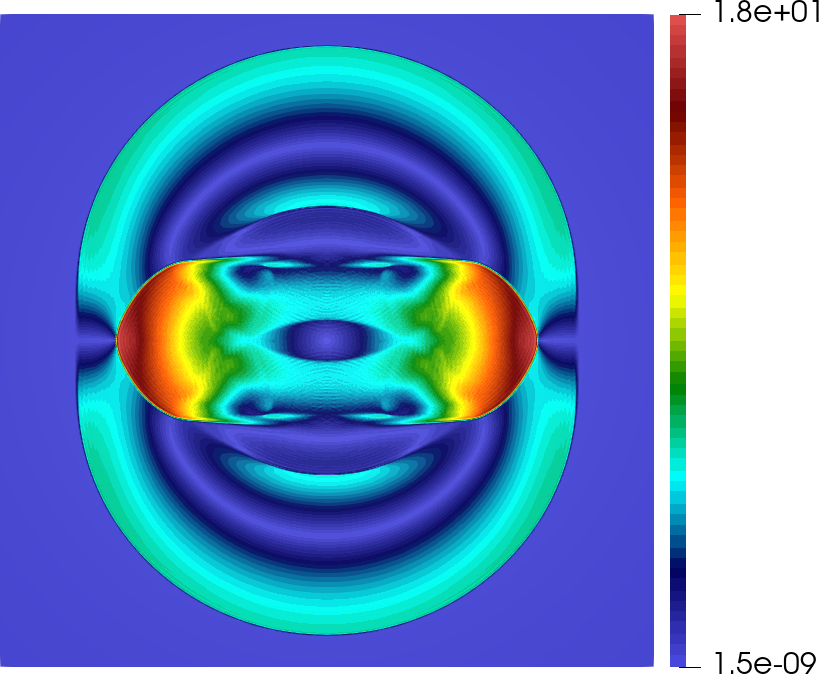}
    \caption{Velocity magnitude $|\vel_h|$}
  \end{subfigure}
  \hfill
  \begin{subfigure}{0.48\textwidth}
    \centering
    \includegraphics[width=\textwidth]{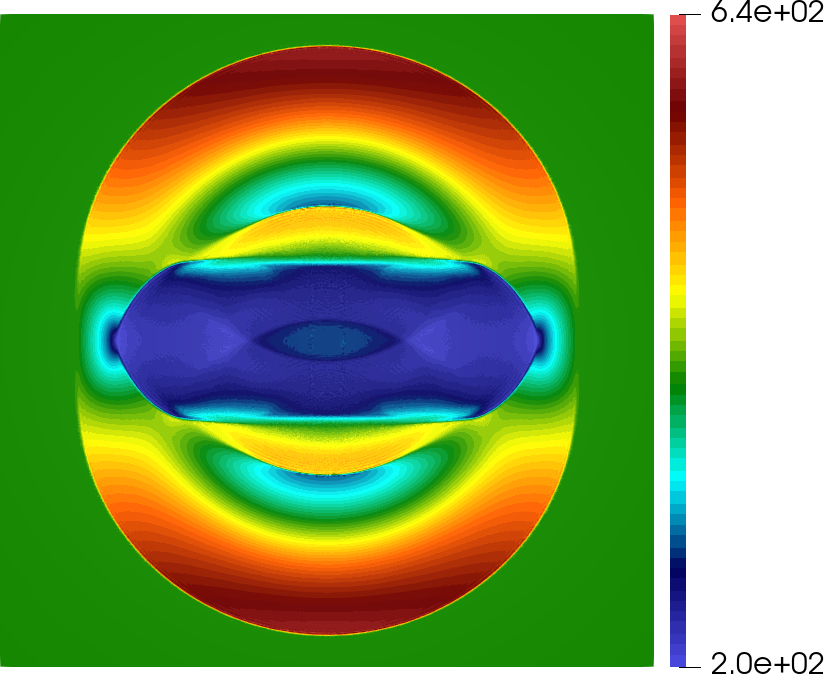}
    \caption{Magnetic pressure $\Bfield_h^2/2$}
  \end{subfigure}
  \hfill
  \caption{Solution to Blast problem at time $t=0.01$ using 290147 nodal points}
  \label{fig:Blast}
\end{figure}

\subsection{Astrophysical jet \cite{Wu2018}}
\label{sec:numerical:jet}

The last benchmark is proposed by \cite{Wu2018}. The domain is
$[-0.5, 0.5]\times [0, 1.5]$. The initial ambient fluid is given by
\begin{align*}
 [\rho, \vel, p, \Bfield]^\top = [0.1\gamma, 0,0, 1.0, 0,
B_a]^\transp.
\end{align*}
A Mach 800 inflow is set on the inlet boundary
$\xcoord \in (-0.05, 0,05) \times \{0\}$:
\begin{align*}
[\rho, \vel] = [\gamma,0, 800]^\transp.
\end{align*}
The adiabatic gas constant $\gamma=1.4$. The solution is simulated on half the
domain and reflecting boundary condition is imposed on the line $x=0$. The Euler
counterpart of this Mach 800 jet is already a very difficult test of positivity
preservation of numerical schemes. Fortunately, we have a good Euler solver
which can overcome this difficulty. Since the magnetic component is not present
in the mechanical pressure, positivity of pressure in the split form is not
interfered by how the magnetic field is. We observe that our method performs
very well regardless of how extreme the magnetic field is. The solution is shown
in Figure~\ref{fig:Jet_200} when setting $B_a=\sqrt{200}$ and
\ref{fig:Jet_20000} when setting $B_a=\sqrt{20000}$. In
Figure~\ref{fig:Jet_200}(b), we notice that the magnetic pressure is sharp but
less smooth in some regions. This can be due to the fact that the magnetic field
is not regularized. Since we did not implement out flow boundary conditions, the
domain is extended to the directions of the bow shocks. The extended domain
parts are cut out from the final plots.

\begin{figure}[h!]
  \centering
  \begin{subfigure}{0.3\textwidth}
    \centering
    \includegraphics[width=\textwidth]{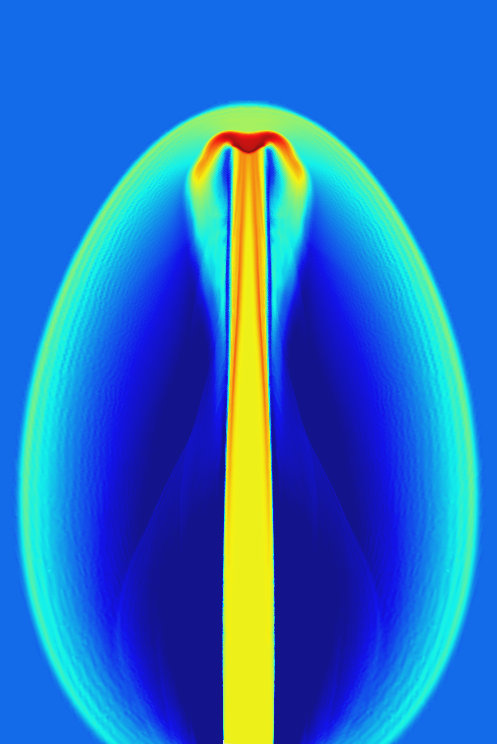}
    \caption{Density logarithm $\log_{10}(\rho_h)$}
  \end{subfigure}
  \begin{subfigure}{0.3\textwidth}
    \centering
    \includegraphics[width=\textwidth]{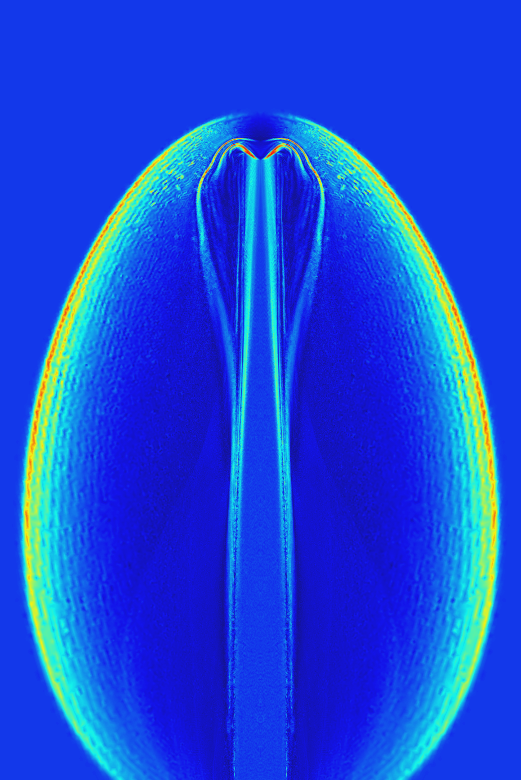}
    \caption{Magnetic pressure $\Bfield_h^2/2$}
  \end{subfigure}
  \begin{subfigure}{0.3\textwidth}
    \centering
    \includegraphics[width=\textwidth]{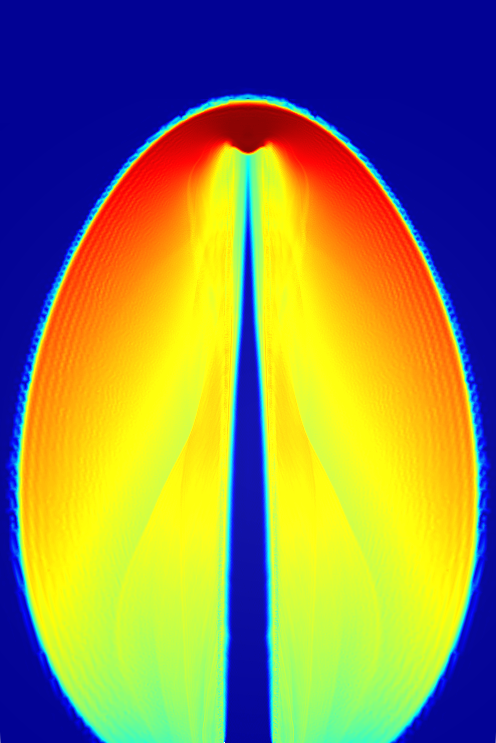}
    \caption{Pressure logarithm $\log_{10}(p_h)$}
  \end{subfigure}
  \caption{Solution to the astrophysical jet problem at time $t=0.002$,
  $B_a = \sqrt{200}$ using 136173 nodal points}
  \label{fig:Jet_200}
\end{figure}

\begin{figure}[h!]
  \centering
  \begin{subfigure}{0.3\textwidth}
    \centering
    \includegraphics[width=\textwidth]{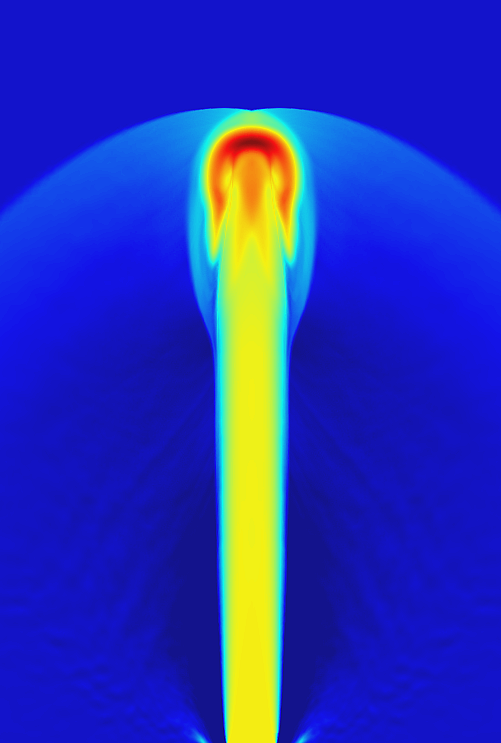}
    \caption{Density logarithm $\log_{10}(\rho_h)$}
  \end{subfigure}
  \begin{subfigure}{0.3\textwidth}
    \centering
    \includegraphics[width=\textwidth]{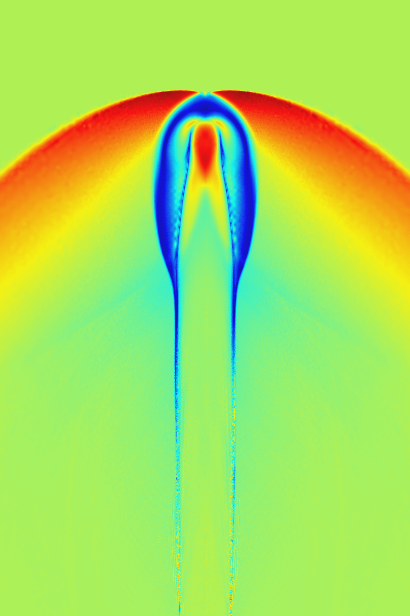}
    \caption{Magnetic pressure $\Bfield_h^2/2$}
  \end{subfigure}
  \begin{subfigure}{0.3\textwidth}
    \centering
    \includegraphics[width=\textwidth]{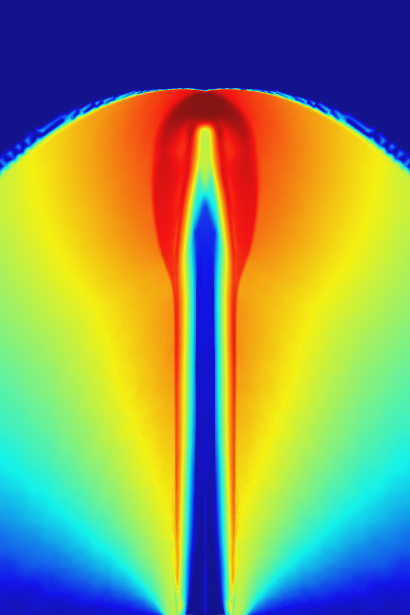}
    \caption{Pressure logarithm $\log_{10}(p_h)$}
  \end{subfigure}
  \caption{Solution to the astrophysical jet problem at time $t=0.002$,
  $B_a = \sqrt{20000}$ using 136173 nodal points}
  \label{fig:Jet_20000}
\end{figure}

\begin{remark}[Solver performance: Why is the scheme competitive?]
Implicit-time integration requires the execution of Newton iterations, with
each iteration involving the inversion a Jacobian. In addition,
the inverse of the Jacobian is applied using a Krylov method, with
each iteration involving two matrix-vector products. The fundamental question
of whether implicit time-integration is competitive (or not) boils down to
having a low count of (linear and nonlinear) iterations. For the
method advanced in this paper we have rather exceptional linear and nonlinear
solver performance. For starters, the nonlinear system \eqref{SourceUpdateCN}
is solved with at most 4 Newton iterations: we hardcoded the logic to
stop the whole computation if more than 4 iterations are needed. This is in big
part because we are using the solution from the previous time-step as an initial
guess. On the other hand, even though the method does not have to respect the
CFL of the MHD system, and magnetosonic waves can be ignored, we still
have to respect the CFL of Euler system. Therefore the time-step sizes
are still moderate and the resulting Jacobian is just a
perturbation of the mass matrix. This means that an inexpensive Krylov method,
such as BiCGStab without any form of preconditioning, can be used in practice.
Usually less than a dozen matrix-vector product are used in order to apply the
inverse of the Jacobian. We believe that the scheme is quite
competitive and the incorporation of matrix-free linear algebra for system
\eqref{SourceUpdateCN} would make the current implementation suitable to
execute three-dimensional computations.
\end{remark}

%%%%%%%%%%%%%%%%%%%%%%%%%%%%%%%%%%%%%%%%%%%%%%%%%%%%%%%%%%%%%%%%%%%%%%%%%%%%%%%% 
%%%%%%%%%%%%%%%%%%%%%%%%%%%%%%%%%%%%%%%%%%%%%%%%%%%%%%%%%%%%%%%%%%%%%%%%%%%%%%%%
%%%%%%%%%%%%%%%%%%%%%%%%%%%%%%%%%%%%%%%%%%%%%%%%%%%%%%%%%%%%%%%%%%%%%%%%%%%%%%%%
%%%%%%%%%%%%%%%%%%%%%%%%%%%%%%%%%%%%%%%%%%%%%%%%%%%%%%%%%%%%%%%%%%%%%%%%%%%%%%%%
%%%%%%%%%%%%%%%%%%%%%%%%%%%%%%%%%%%%%%%%%%%%%%%%%%%%%%%%%%%%%%%%%%%%%%%%%%%%%%%%

\appendix

\section{Hyperbolic solver used in this paper}\label{sec:HypSolver}

This section provides a brief outline of the numerical methods used to solve
Euler's system for all the computational experiments advanced in Section
\ref{sec:numexp}. The Euler's system solver presented in this section fills in
the role of $\texttt{euler\_system\_update}$ invoked in Algorithm
\ref{MHDupdateAlg}, which is supposed to comply with the properties described in
Section \ref{sec:hypminimal}. This section does not introduce any novel concept,
idea, or numerical scheme, and it is only provided for the sake of completeness.
The main ideas advanced in this section were originally developed in the
sequence of papers \cite{Guermond_Popov_2016, Guermond_Popov_Fast_Riemann_2016,
Guer2018} and references therein.

\subsection{Low-order scheme}

Low-order scheme is obtained using first-order Graph Viscosity method suggested
first in \cite{Guermond_Popov_2016}. Let $t_n$ be the current time, $\tau_n$ is
the current time-step and we advance in time by setting $t_{n+1} = t_n +
\tau_n$. Let $\state_h^n = \sum_{i \in \HypVertices} \state_i^n
\phi_i(\xcoord)$ be the finite element approximation at time $t_n$. The first
order approximation at time $t_{n+1}$ is computed as
\begin{align}\label{eq:FO}
  m_i\frac{\state_i^{\text{L},n+1} - \state_i^n}{\tau_n}
  + \sum_{j \in \mathcal{I}(i)} \flux(\state_j^n) \bv{c}_{ij}
  - d_{ij}^\text{L} (\state_j^n - \state_i^n) = \bzero ,
\end{align}
where $m_i$ and $\bv{c}_{ij}$ were defined in \eqref{schemeMatrices}, while the
set $\mathcal{I}(i)$ is the so-called stencil, which is defined as
$\mathcal{I}(i) = \big\{ j \in \HypVertices \,  | \, \bv{c}_{ij} \not =
\bzero \big\}$,
while the low-order graph-viscosity $d_{ij}^\text{L}$ is computed as
\begin{align}
% \begin{aligned}
\label{dijLow}
d_{ij}^\text{L} &:=
\max(
\lambda_{\max}(\state_i^n, \state_j^n, \normal_{ij})\,
|\bv{c}_{ij}|_{\ell_2},\
\lambda_{\max}(\state_j^n, \state_i^n, \normal_{ji})\,
|\bv{c}_{ji}|_{\ell_2}), \quad \forall i\not=j, \\
\nonumber
d_{ii}^\text{L} &:= - \sum_{i\not=j \in \HypVertices} d_{ji}^\text{L}
\ \ \text{and} \ \
\normal_{ij} = \frac{\bv{c}_{ij}}{|\bv{c}_{ij}|_{\ell_2}}.
\end{align}
Here $\lambda_{\max}(\state_L, \state_R, \normal)$ is the maximum wave speed of
the one dimensional Riemann problem: $\partial_t \state + \partial_x
(\bv{f}(\state) \normal) = 0$, where $x = \xcoord \cdot \normal$, with initial
condition: $\state(x,0) = \state_L = [\rho_L, \mom_L, \totme_L]^\transp$ if
$x<0$, and $\state(x,0) =
\state_R = [\rho_R, \mom_R, \totme_R]^\transp$ if $x\geq 0$. The maximum
wavespeed of this Riemann problem can be
computed exactly \cite[Chap.~4]{Toro_2009}, however this comes at the
expense of solving a nonlinear problem. In theory and practice, any upper
bound of the maximum wavespeed of the Riemann problem could be used in
formula \eqref{dijLow} while still preserving rigorous mathematical
properties of the scheme \cite{Guermond_Popov_2016, Guer2019}. For the specific
case of the covolume equation of state; $p(1-b\rho)=(\gamma-1)e\rho$ with $b\ge
0$; we can use $\lambda^{\#}(\state_L, \state_R, \bv{n})$ which is defined by
\begin{align}
% \label{LambdaSharp}
\nonumber
&\lambda^{\#}(\state_L, \state_R, \bv{n}) =
\max((\lambda_1^-(p^{\#}) )_{-},\ (\lambda_3^+(p^{\#}))_+), \\
\nonumber
&\lambda_1^-(p^{\#})=v_L - c_L\left(
1+\frac{\gamma+1}{2\gamma}\left(\frac{p^{\#}-p_L}{p_L}\right)_+
\right)^\frac12, \\
\nonumber
&\lambda_3^+(p^{\#})=v_R + c_R\left(
1+\frac{\gamma+1}{2\gamma}\left(\frac{p^{\#}-p_R}{p_R}\right)_+
\right)^\frac12, \\
\label{PSharp}
&p^{\#} := \left(
\frac{c_L(1-b\rho_L)+c_R(1-b\rho_R)-\frac{\gamma-1}{2}(v_R-v_L)}
{ c_L(1-b\rho_L)\ p_L^{-\frac{\gamma-1}{2\gamma}}
+ c_R(1-b\rho_R)\ p_R^{-\frac{\gamma-1}{2\gamma} } }
\right)^{\frac{2\gamma}{\gamma-1}},
\end{align}
where $z_-:=\max(0,-z)$, $z_+:=\max(0,z)$, $v_L = \vel_L \cdot \normal$, $v_R =
\vel_R \cdot \normal$, $p_L$ and $p_R$ are the left and right pressures,
and $c_L$ and $c_R$ are left and right sound speeds. Here the formula
\eqref{PSharp} is often referred to as the two-rarefaction estimate
\cite{Toro_2009}. It is possible to show that $\lambda_{\max}(\state_L,
\state_R, \bv{n}) \leq \lambda^{\#}(\state_L, \state_R, \bv{n})$
\cite{Guermond_Popov_Fast_Riemann_2016} for $1 < \gamma \le \frac53$. For
all computations presented in this paper, $\lambda^{\#}(\state_L,
\state_R, \bv{n})$ is used instead of $\lambda_{\max}(\state_L,
\state_R, \normal)$ in order to compute the algebraic viscosities described in
\eqref{dijLow}. We finally mention that scheme
\eqref{eq:FO} equipped with the viscosity \eqref{dijLow}, is compatible with
the assumption \eqref{EulerSolverDissipation}. \\

\begin{remark}[Convex reformulation and CFL condition] The scheme
\eqref{eq:FO} can be rewritten as
\begin{align}
\label{ConvexReform}
&\state_i^{n+1} = \Big(1 - \sum_{j \in \mathcal{I}(i)\backslash\{i\}}\tfrac{2
\dt_n d_{ij}^{\low,n}}{m_i}\Big) \,
\state_{i}^{n}
+ \sum_{j \in \mathcal{I}(i)\backslash\{i\}}
\Big(\tfrac{2\dt_n d_{ij}^{\low,n}}{m_i}\Big) \,\overline{\state}_{ij}^{n},
\end{align}
where
\begin{align}
\label{BarState}
\overline{\state}_{ij}^{n} = \tfrac{1}{2}(\state_j^{n} + \state_i^{n})
- \tfrac{|\bv{c}_{ij}|}{2 d_{ij}^\low} (\flux(\state_j^{n}) -
  \flux(\state_i^{n})) \normal_{ij}
\end{align}
are the so-called
bar-states. We note that the states $\{\overline{\state}_{ij}^{n}\}_{j \in
\mathcal{I}(i)}$ are admissible provided that $\state_i^{n}$ and $\state_j^{n}$
are admissible and that $d_{ij}^\low \geq
\max(
\lambda_{\max}(\state_i^n, \state_j^n, \normal_{ij})\,
\|\bv{c}_{ij}\|_{\ell_2},\
\lambda_{\max}(\state_j^n, \state_i^n, \normal_{ji})\,
\|\bv{c}_{ji}\|_{\ell_2})$, see \cite{Guermond_Popov_2016, Guer2019}. We note
that $\state_i^{n+1}$ is a convex combination of the bar-states
$\{\overline{\state}_{ij}^{n}\}_{j \in \mathcal{I}(i)}$ provided the condition
$\Big(1 - \sum_{j \in \mathcal{I}(i)\backslash\{i\}}\tfrac{2
\dt_n d_{ii}^{\low,n}}{m_i}\Big) \geq 0$ holds. Therefore, we define the largest
admissible time-step size as
\begin{align*}
  \dt_n = \text{CFL} \cdot \min_{i \in \HypVertices}
  \big(-\tfrac{m_i}{2 d_{ii}^{\low,n}}\big)
\end{align*}
where $\text{CFL} \in (0,1)$ is a user defined parameter.
\end{remark}

\subsection{High-order scheme}

We note that the scheme \eqref{eq:FO} can only be first-order accurate.
Therefore we consider the high-order scheme:
\begin{align}\label{eq:HighMethod}
  \sum_{j \in \mathcal{I}(i)} m_{ij}\frac{\state_j^{\text{\high},n+1} -
  \state_j^n}{\tau_n}
  + \sum_{j \in \mathcal{I}(i)} \flux(\state_j^n) \bv{c}_{ij}
  - d_{ij}^{\high} (\state_j^n - \state_i^n) = \bzero ,
\end{align}
Here $\{d_{ij}^\high\}_{j \in \mathcal{I}(i)}$ are the high-order viscosities
which are meant to be such that $d_{ij}^{\high} \approx 0$ in smooth regions of
the domain, while $d_{ij}^\high \approx d_{ij}^\low$ near shocks and
discontinuities. In addition, $d_{ij}^\high$ must be symmetric and conservative,
\ie $d_{ij}^\high = d_{ji}^\high$ and
$d_{ii}^\high := - \sum_{i\not=j \in
\HypVertices} d_{ji}^\high$.

In this paper, we use a high-order viscosity that is proportional to the
entropy residual (i.e. entropy-production) of the unstabilized scheme. Let us
start by considering the Galerkin solution $\state_h^{\text{G}}$ defined as
\begin{align}\label{eq:gal}
m_{i}\frac{\state_i^{\text{G}} -
\state_i^n}{\tau_n}
+ \sum_{j \in \mathcal{I}(i)} \flux(\state_j^n) \bv{c}_{ij}
= \bzero
\ \text{ for all }i \in \HypVertices.
\end{align}
Let $\{\eta(\state), \eflux(\state) \}$ be an entropy pair of the Euler system.
We define the entropy residual function $R^n_h(\state_h) = \sum_{i \in
\HypVertices} R_i^n \HypBasisComp_i \in \FESpaceHypComp$ with nodal values
$R_i^n$ defined by
\begin{align}\label{RiDefOne}
R_i^n:=
m_{i} \frac{\state_i^{\text{G}} -
\state_i^n}{\tau_n} \cdot \nabla_{\state}\eta(\state_i^n)
+ \sum_{j \in \mathcal{I}(i)} \eflux(\state_j^n) \bv{c}_{ij}
\ \text{ for all }  i \in \HypVertices.
\end{align}
Here $R_i^n$ is proportional to the entropy production of the unstabilized
scheme \eqref{eq:gal}. However, formula \eqref{RiDefOne} is not practical, since
it requires computing $\state_i^{\text{G}}$. We derive a
formula for $R_i^n$ that does not invoke $\state_i^{\text{G}}$: multiplying
\eqref{eq:gal} by $\nabla_{\state}\eta(\state_i^n)$ we get that
\begin{align*}
m_{i}\frac{\state_i^{\text{G}} -
\state_i^n}{\tau_n}\cdot\nabla_{\state}\eta(\state_i^n) =
- \sum_{j \in \mathcal{I}(i)} (\flux(\state_j^n)
\bv{c}_{ij})\cdot\nabla_{\state}\eta(\state_i^n)
\end{align*}
which we use to replace the first term in \eqref{RiDefOne}:
\begin{align}\label{eq:entres}
R_i^n=
\sum_{j \in \mathcal{I}(i)}
- (\flux(\state_j^n) \bv{c}_{ij})\cdot\nabla_{\state}\eta(\state_i^n)
+ \eflux(\state_j^n) \bv{c}_{ij}
\text{ for all } i \in \HypVertices.
\end{align}
In practice, we use \eqref{eq:entres} in order to compute the entropy-viscosity
indicators. We are now ready to define the high-order nonlinear viscosity as
\begin{align*}
  d_{ij}^\high:= \min
  \Big(
  d_{ij}^\low, c_{\text{EV}}
  \max
  (
  \widetilde R^n_i, \widetilde R^n_j
  )  
  \Big),
\end{align*}
where $c_{\text{EV}}$ is a tunable constant, which is taken to be equal to 1 in
numerical examples in this manuscript, and $\widetilde R^n_i$ is the normalized
entropy residual:
\begin{align*}
  \widetilde R^n_i := \frac{R^n_i}{
    \max
    \Big(
    \rho^{\max,\, n}_i s^{\max,\, n}_i - \rho^{\min,\, n}_i s^{\min,\, n}_i
    ,
    \epsilon \|\eta^n_h\|_{L^\infty(\Omega)}
    \Big)
  },
\end{align*}
where $w^{\max,\, n}_i:=\max_{j\in \mathcal{I}(i)}w^n_j$, and
$
w^{\min,\, n}_i:=\min_{j\in \mathcal{I}(i)}w^n_j,
$
for $w$ being $\rho$ or $s$. Recall that the mathematical entropy is computed as
$\eta(\state)= - \rho s(\state)$, where $s(\state) =\frac{1}{\gamma - 1}
\log(e) - \log(\rho)$ is the specific entropy. A small safety factor
$\epsilon=10^{-8}$ is used to avoid division by zero.

\subsection{Convex limiting}

The low-order and high-order methods can be convenient rewritten as
\begin{align*}
  m_i \big(\state_i^{\low,n+1} - \state_i^n\big)
  + \sum_{j \in \mathcal{I}(i)} \bv{F}_{ij}^{\low} = \bzero
  \ \ \ \text{and} \ \ \
  m_i \big(\state_i^{\high,n+1} - \state_i^n\big)
  + \sum_{j \in \mathcal{I}(i)} \bv{F}_{ij}^{\high} = \bzero
\end{align*}
where the algebraic fluxes $\bv{F}_{ij}^{\low}$ are defined as
$\bv{F}_{ij}^{\high}$
\begin{align*}
  \bv{F}_{ij}^{\low} &=
                       \dt_n \big(\flux(\state_j^n) + \flux(\state_i^{n})\big)\bv{c}_{ij} -
                       \dt_n d_{ij}^{\low,n} (\state_j^n - \state_i^n), \\
  \bv{F}_{ij}^{\high} &=
                        \dt_n \big(\flux(\state_j^n) +
                        \flux(\state_i^n) \big)\bv{c}_{ij}
                        - \dt_n d_{ij}^{\high} (\state_j^{n} - \state_i^{n}) \\
                     &\qquad + (m_{ij} - \delta_{ij} m_i) (\state_j^{\high,n+1} -
                       \state_j^n
                       - \state_i^{\high, n+1} + \state_i^n).
\end{align*}
We define the algebraic flux-corrections $\bv{A}_{ij} = \bv{F}_{ij}^{\low} -
\bv{F}_{ij}^{\high}$ and set the final flux-limited solution to be
\begin{align}
\label{FluxLimitedScheme}
\begin{split}
m_i \state_i^{n+1}
=
m_i \state_i^{\low, n+1}
+ \sum_{j \in \mathcal{I}(i)} \limiter_{ij} \bv{A}_{ij}
\end{split}
\end{align}
where $\limiter_{ij} \in [0,1]$ are the limiters. If $\limiter_{ij} \equiv 0$
for all $i$ and $j$, then \eqref{FluxLimitedScheme} recovers $\state_i^{n+1} =
\state_i^{\low, n+1}$. Similarly, if $\limiter_{ij} \equiv 1$
for all $i$ and $j$ then \eqref{FluxLimitedScheme} leads to $\state_i^{n+1} =
\state_i^{\high, n+1}$. The goal is to select limiters as large as it could be
possible while also preserving important bounds.

We want to enforce local bounds on the density and local minimum principle of the
specific entropy. However, logarithmic entropies, such as $s(\state) = \ln
\frac{p(\state)}{\rho^\gamma}$, are not friendly in the context
Newton-like line search iterative methods. Therefore, we use $\tilde{s}(\state)
= \rho^{-\gamma} \inte(\state)$ which leads to an entirely equivalent minimum
principle since
\begin{align*}
s(\state) \leq s(\boldsymbol{v}) \ \Leftrightarrow \
\tilde{s}(\state) \leq \tilde{s}(\boldsymbol{v})
\ \text{for all} \ \state, \boldsymbol{v} \in \mathcal{A} \, ,
\end{align*}
due to the monotonicity of $\ln x$. Therefore, at each node $i \in
\HypVertices$ we compute the bounds:
\begin{align*}
\rho_i^{\text{min}} &:= \mathbb{1}_h^{-} \min_{j \in \mathcal{I}(i)}
\min \{\rho_j^n, \overline{\rho}_{ij}^n\} \\
\rho_i^{\text{max}} &:= \mathbb{1}_h^{+} \max_{j \in \mathcal{I}(i)}
\max \{\rho_j^n, \overline{\rho}_{ij}^n\} \\
\tilde{s}_i^{\text{min}} &:= \mathbb{1}_h^{-} \min_{j \in \mathcal{I}(i)}
\min \{\tilde{s}_j^n, \overline{\tilde{s}}_{ij}^n\}
\end{align*}
where $\overline{\rho}_{ij}^n$ denotes the density of the bar-state
$\overline{\state}_{ij}^{n}$ (see expression \eqref{BarState}), while
$\overline{\tilde{s}}_{ij}^n := \tilde{s}(\overline{\state}_{ij}^{n})$. Here
$\mathbb{1}_h^{-}$ and $\mathbb{1}_h^{+}$ are just ad-hoc relaxations of the
unity with a prescribed decay rate with respect to the local meshsize $h$. More
precisely we consider
\begin{align*}
\mathbb{1}_h^{-} = 1 - \kappa  (\tfrac{m_i}{|\Omega|})^\frac{p}{d}
\ \ \text{and} \ \
\mathbb{1}_h^{+} = 1 + \kappa (\tfrac{m_i}{|\Omega|})^\frac{p}{d}
\ \ \text{with} \ p = 1.50, \ d = 2.0, \ \text{and} \ \kappa = 4.0 \, .
\end{align*}
We mention in passing that, asymptotically for $h \rightarrow 0$, the value of
$\kappa$ is has no importance and we may use any other $\kappa =
\mathcal{O}(1)$. At each node $i \in \HypVertices$ we define the set
\begin{align*}
\mathcal{B}_i = \big\{ \state = [\rho,\mom, \totme]^\transp \in \mathbb{R}^{d+2}
\, \big| \,
\rho_i^{\text{min}} \leq \rho \leq \rho_i^{\text{max}} ,
\tilde{s}(\state) \geq \tilde{s}_i^{\text{min}} \big\}
\end{align*}
We note that \eqref{FluxLimitedScheme} can be conveniently rewritten as
\begin{align}
\label{ConvexScheme}
\begin{split}
\state_i^{\high, n+1}
= \sum_{j \in \mathcal{I}(i)}
\lambda_i (\state_i^{\low, n+1} + \ell_{ij} \bv{P}_{ij})
\ \ \text{where} \ \
\lambda_i = \tfrac{1}{\text{card}\mathcal{I}(i) - 1}
\ \ \text{and} \ \
\bv{P}_{ij} = \tfrac{1}{\lambda_i m_i} \bv{A}_{ij}
\end{split}
\end{align}
Convex-limiting is built on the observation that condition $\state_i^{\high,
n+1} \in \mathcal{B}_i$ will hold if $\state_i^{\low, n+1} + \ell_{ij}
\bv{P}_{ij} \in \mathcal{B}_i$ for all $j \in \mathcal{I}(i)$. Therefore, at
each node $i$ we compute the preliminary limiters $l_{ij}$ as
\begin{align*}
l_{ij} :=
\texttt{compute\_line\_search}(\state_i^{\low, n+1}, \bv{P}_{ij},
\rho_i^{\text{min}}, \rho_i^{\text{max}}, \tilde{s}_i^{\text{min}})
\end{align*}
with $\texttt{compute\_line\_search}$ as defined in Algorithm
\ref{ComputeLineSearchAlg}, while the final limiters are computed as
$\ell_{ij} = \min \{l_{ij}, l_{ji}\}$ in order to guarantee conservation
properties of the scheme, see \cite{Guer2018, Guer2021, Kronbichler2021} for
both theory and implementation detail.

\begin{algorithm}[H]
\caption{\texttt{compute\_line\_search}($\state, \bv{P},
\varrho^\textit{min}, \varrho^\textit{max}, \tilde{s}^\textit{min}$) }
\label{ComputeLineSearchAlg}
\begin{align*}
&\ell^{\rho,\textit{min}} :=
\max \{\ell \in [0,1] \, | \, \rho(\state + \ell \bv{P}) \geq
\varrho^{\textit{min}} \} \\
&\ell^{\rho,\textit{max}} := \max \{\ell \in
[0,\ell^{\rho,\textit{min}}]
\, | \,
\rho(\state + \ell \bv{P}) \leq \varrho^{\textit{max}} \} \\
&\ell^{s} := \max \{\ell \in [0,\ell^{\rho,\textit{max}}]
\, | \,
\tilde{s}(\state + \ell \bv{P}) \geq
\tilde{s}^\textit{min} \} \\
&\texttt{Return:} \ell^{s}
\end{align*}
\textit{Comments: input arguments are $\state, \bv{P} \in
\mathbb{R}^{m}$ and $\varrho^\textit{min}, \varrho^\textit{max},
\tilde{s}^\textit{min} \in \mathbb{R}^+$.}
\end{algorithm}

\bibliographystyle{plain}
\bibliography{MacrosAndBiblio/biblio}

\end{document}